\documentclass[namedate,webpdf,imanum]{ima-authoring-arxiv}%

\theoremstyle{thmstyletwo}%
\newtheorem{theorem}{Theorem}%
\newtheorem{proposition}{Proposition}%
\newtheorem{lemma}{Lemma}

\newtheorem{definition}{Definition}
\newtheorem{assumption}{Assumption}

\newtheorem{remark}{Remark}%
\numberwithin{equation}{section}

\renewcommand{\proofname}{Proof.\;}

\DeclareOldFontCommand{\rm}{\normalfont\rmfamily}{\mathrm}

\usepackage{tikz}
\usetikzlibrary{positioning,shapes,calc,3d}
\usepackage{pgfplots}
\usepgfplotslibrary{fillbetween}

\pgfplotsset{%
every axis/.append style={width=.45\textwidth,
                              axis x line=bottom, axis y line=left,
                              x axis line style={thick,->}, y axis line style={thick,->},
                              tick align=inside, tick style={thick},
                              every x tick label/.style={font=\footnotesize},
                              every y tick label/.style={font=\footnotesize},
                              },
every axis legend/.append style={
                              legend columns=1,
                              font=\footnotesize,
                              draw=none,
                              fill=white,
                              },
every axis x label/.style={at={(0.5,0.1)},below,fill=none,fill opacity=1,text opacity=1},
every axis y label/.style={at={(0.1,0.5)},fill=none,fill opacity=1,text opacity=1,rotate=90},
compat=newest,
}

\newcommand{\rd}{\, {\rm d}}
\newcommand{\wal}{{\rm wal}}
\newcommand{\bsx}{\boldsymbol{x}}
\newcommand{\bsy}{\boldsymbol{y}}

\newcommand{\bsk}{\boldsymbol{k}}
\newcommand{\bsm}{\boldsymbol{m}}
\newcommand{\bsq}{\boldsymbol{q}}
\newcommand{\bsell}{\boldsymbol{\ell}}

\newcommand{\bsgamma}{\boldsymbol{\gamma}}

\newcommand{\bsnu}{\boldsymbol{\nu}}
\newcommand{\bszero}{\boldsymbol{0}}
\newcommand{\bsone}{\boldsymbol{1}}
\newcommand{\bstwo}{\boldsymbol{2}}
\newcommand{\bsdelta}{\boldsymbol{\delta}}
\newcommand{\stu}{\mathrm{stu}}
\newcommand{\rat}{\mathrm{rat}}

\newcommand{\uu}{\mathfrak{u}}
\newcommand{\vv}{\mathfrak{v}}
\newcommand{\ww}{\mathfrak{w}}

\newcommand{\NN}{\mathbb{N}}
\newcommand{\RR}{\mathbb{R}}
\newcommand{\EE}{\mathbb{E}}

\newcommand{\cP}{\mathcal{P}}

\allowdisplaybreaks

\title[Second order interlaced polynomial lattice rules]{Second order interlaced polynomial lattice rules for integration over $\mathbb{R}^s$}

\author{Tiangang Cui*\ORCID{0000-0002-4840-8545}
\address{\orgdiv{School of Mathematics and Statistics}, \orgname{The University of Sydney}, \orgaddress{\street{Camperdown}, \postcode{2006}, \state{NSW}, \country{Australia}}}}
\author{Josef Dick\ORCID{orcid.org/0000-0003-0142-6022}
\address{\orgdiv{School of Mathematics and Statistics}, \orgname{UNSW Sydney}, \orgaddress{\street{Kensington}, \postcode{2052}, \state{NSW}, \country{Australia}}}}
\author{Friedrich Pillichshammer\ORCID{0000-0001-6952-9218}
\address{\orgdiv{Institut f\"{u}r Finanzmathematik und Angewandte Zahlentheorie}, \orgname{Johannes Kepler Universit\"{a}t Linz}, \orgaddress{\street{Altenbergerstra{\ss}e 69}, \postcode{A-4040}, \state{Linz}, \country{Austria}}}}

\authormark{T. Cui, J. Dick, F. Pillichshammer}
\corresp[*]{Corresponding author: \href{email:email-id.com}{friedrich.pillichshammer@jku.at}}
\received{Date}{0}{Year}
\revised{Date}{0}{Year}
\accepted{Date}{0}{Year}

\copyrightyear{}

\begin{document}

\abstract{
We study numerical integration of functions $f: \mathbb{R}^{s} \to \mathbb{R}$ with respect to a probability measure. By applying the corresponding inverse cumulative distribution function, the problem is transformed into integrating an induced function over the unit cube $(0,1)^{s}$. We introduce a new orthonormal system: \emph{order~2 localized Walsh functions}. These basis functions retain the approximation power of classical Walsh functions for twice‑differentiable integrands while inheriting the spatial localization of Haar wavelets. Localization is crucial because the transformed integrand is typically unbounded at the boundary. We show that the worst‑case quasi‑Monte Carlo integration error decays like $\mathcal{O}(N^{-1/\lambda})$ for every $\lambda \in (1/2,1]$. 
As an application, we consider elliptic partial differential equations with a finite number of log‑normal random coefficients and show that our error estimates remain valid for their stochastic Galerkin discretizations by applying a suitable importance sampling density. 
}


\maketitle

\section{Introduction}

Quasi-Monte Carlo (QMC) methods have become indispensable tools for the efficient numerical approximation of high‐dimensional integrals arising in uncertainty quantification, financial mathematics, and statistical physics. By replacing independent random sampling with carefully constructed deterministic point sets—so‐called low‐discrepancy sequences—QMC rules often deliver dramatically improved convergence rates compared to ordinary Monte Carlo, without sacrificing robustness in high dimensions for integrands whose input on later coordinates diminishes. In classical QMC theory, much of the focus has been on integrals over the unit cube $[0,1]^s$ against the uniform probability measure, where, under sufficient smoothness of the integrand, higher‐order convergence rates of order $O(N^{-\alpha}(\log N)^{(s-1)\alpha})$ (for some $\alpha\ge0$) can be achieved using higher order digital nets or interlaced polynomial lattice rules \citep{DP10}, where $N$ is the number of employed points in the QMC rule. In particular, for sufficiently smooth integrands on $[0,1]^s$, order‐two digital nets attain an $O(N^{-2+\delta})$ convergence rate for any $\delta > 0$ with dimension‐independent constants under appropriately chosen weights which moderate the influence of subsets of input variables \citep{dick08, G15}.

However, many applications require integration over the unbounded domain $\mathbb{R}^s$ with respect to a product probability density (e.g., Gaussian or log‐Gaussian) of the form
\begin{equation}\label{def:int}
\int_{\mathbb{R}^s} F(\bsx) \varphi(\bsx) \,\mathrm{d} \bsx,
\end{equation}
where 
$\varphi(\boldsymbol{x})=\prod_{j=1}^s\varphi(x_j)$ is a product density on $\mathbb{R}^s$ (e.g., a standard Gaussian or other probability density function (PDF) with well‐defined derivative) and $F$ belongs to a normed space of sufficiently smooth functions on $\mathbb{R}^s$. Integrals of the form \eqref{def:int} arise, for example, when computing expected values of output functionals of the (deterministic) PDE solution $u(\cdot,\boldsymbol x)$ with log‐normal diffusion coefficient $a(\cdot,\boldsymbol x)=\exp(Z(\cdot,\boldsymbol x))$, where $Z$ is a Gaussian random field parameterized by the vector of random variables $\boldsymbol x\in\mathbb{R}^s$ via a truncated Karhunen--Loève expansion or in computing the expected value of option prices in financial mathematics.

A QMC theory for lattice rules was developed in \citet{NK14}, who gave a fast component‐by‐component (CBC) construction of randomly shifted lattice rules that achieve $O(N^{-1+\delta})$ for integrands in weighted Sobolev spaces on $\mathbb{R}^s$—again with dimension‐independent constants—by combining Gaussian‐measure taming with suitable product‐and‐order‐dependent (POD) weights. \citet{K19} further developed lattice‐based QMC theory for log‐normal PDEs, again obtaining $O(N^{-1+\delta})$ quadrature error. More recent nonasymptotic analyses, e.g., \cite{LT23}, refine these estimates by explicitly tracking dependence on the finite number of QMC points and the dimension, but still only establish first‐order convergence in practice. Numerical integration for PDEs with log-normal coefficients was studied in \cite{HS19}. In \cite{NN21}, a multivariate decomposition method for PDEs with log-normal random coefficients was analyzed. Higher order convergence rates for QMC integration over unbounded domains were proven in \cite{DILP18}, but with a very poor dependence of the error bound on the dimension. Another approach has been proposed in \cite{NS23}, where higher order QMC methods for integration over $\mathbb{R}^s$ were obtained using lattice rules on a truncated domain together with periodization. For general quadrature rules achieving the optimal order of convergence rates (i.e. including the best possible exponent of the $\log N$ factor in the upper bound), see \cite{DN24}. An approach based on avoiding the boundary of the unit cube after the transformation, as studied in \cite{O06}, together with a smoothing procedure has been studied in \cite{OWH24}. First order numerical integration using digital nets has been studied in \cite{DP24}. The dimension dependence of numerical integration of Hermite spaces over $\mathbb{R}^s$ has been studied in \cite{LPE23}. In general, it is a difficult problem to obtain optimal convergence rates beyond $1$ and dimension-independent bounds. For instance, the suboptimality of Gauss-Hermite rules for integration over $\mathbb{R}$ has been shown in \cite{KSG23}, and hence sparse grid methods based on Gauss-Hermite rules also cannot yield optimal convergence rates for integration over $\mathbb{R}^s$.

Here we construct and analyze second order digital nets on $[0,1)^s$ transformed to $\mathbb{R}^s$ using an inverse normal CDF via a suitable localized Walsh bases representation. We establish convergence rates of the numerical integration error of order $N^{-2 + \delta}$. 

The outline of this paper is as follows. In the forthcoming Section~\ref{sec:lwf}, we introduce the novel concept of localized Walsh functions, which serve as a crucial analytical tool in our study. These functions themselves merit attention, as they combine the features of classical Walsh functions and Haar functions: The Walsh coefficients of twice differentiable functions defined on $[0,1)$ of the localized Walsh functions decay as fast as those of classical Walsh coefficents (as shown in Section~\ref{sec:serexp}), and on the other hand, the localized Walsh functions have shrinking support similar to Haar functions. In Remark~\ref{rem_w_local} below we point out the technical obstruction which prevents us from using classical Walsh functions. On the other hand, we also cannot use Haar functions directly as the Haar coefficients do not show an improved rate of convergence for twice differentiable functions, as is true for the Walsh coefficients (of classical and localized Walsh functions). In Section~\ref{sec:serexp}, we also analyze the series expansion of functions on $\mathbb{R}^s$ with respect to those localized Walsh functions. Assuming certain conditions on the weight functions $\varphi$, $\psi$, and $\omega$—which we collectively denote by Assumption~\ref{assumptionA} and which appear in the definitions of the norm \eqref{Fnorm:infty}—we derive sharp estimates for the corresponding Walsh coefficients. In Section~\ref{sec:intPLPS}, we apply these coefficient bounds to study the worst-case integration error of QMC rules. Specifically, we employ interlaced polynomial lattice point sets of order two, yielding our main result stated in Theorem~\ref{thm:error_bound}. In Section~\ref{sec:ex:ass:A}, we present significant examples of weight functions $\varphi$, $\psi$, and $\omega$ that satisfy Assumption~\ref{assumptionA}. Here we focus on integration with respect to the normal distribution. Section~\ref{sec:Appl} then applies our theoretical findings to an elliptic PDE with a finite number of log-normal random coefficients. For the resulting integration problem, we use importance sampling in conjunction with Theorem~\ref{thm:error_bound}, obtaining—under suitable hypotheses—an integration error of order $N^{-2+\delta}$. Finally, numerical experiments are presented in Section~\ref{sec:numerics}.

\section{Preliminaries}

For a PDF $\varphi:\RR \rightarrow [0,\infty)$ let
\begin{equation}\label{def:Phi}
\Phi(x) := \int_{-\infty}^x \varphi(t) \rd t, \quad \mbox{for } x \in \mathbb{R}.
\end{equation}
be the corresponding cumulative distribution function (CDF). We assume that the inverse function $\Phi^{-1}:(0,1) \to \mathbb{R}$ of $\Phi$ is well defined. For vector arguments, we apply $\Phi$ and $\Phi^{-1}$ coordinate-wise. We approximate the integral \eqref{def:int} by a QMC rule of the form
\begin{equation}\label{QMC:rule}
\frac{1}{N} \sum_{n=0}^{N-1} F(\Phi^{-1}(\bsx_n)),
\end{equation}
where $\mathcal{P}=\{\bsx_0,\bsx_1,\ldots,\bsx_{N-1}\}$ is a set of $N$ points in $(0,1)^s$. The accuracy of this approximation is measured by the absolute integration error $|\mathrm{err}(F; \varphi; \mathcal{P})|$, where
\begin{equation*}
\mathrm{err}(F; \varphi; \mathcal{P}) := \int_{\mathbb{R}^s} F(\bsx) \varphi(\bsx) \rd \bsx - \frac{1}{N} \sum_{n=0}^{N-1} F(\Phi^{-1}(\bsx_n)).
\end{equation*}

We consider functions $F: \RR^s \rightarrow \RR$ with finite $\bsgamma$-weighted $q$-norm, defined as follows. For $\emptyset \not= \uu \subseteq [s]$ and for $\vv \subseteq \uu$ put
\begin{equation}\label{eq_uvnorm}
\|F\|_{\uu, \vv} :=  \sum_{\ww \subseteq \vv} \int_{\mathbb{R}^{|\uu|}} \left| \int_{\mathbb{R}^{s-|\uu|}} \frac{\partial^{|\uu| + |\ww|} F}{\partial \bsx_\ww \partial \bsx_{\uu}}(\bsx)\prod_{j \in [s]\setminus \uu} \varphi(x_j) \rd \bsx_{[s] \setminus \uu} \right|   \prod_{j \in \ww} \psi(x_j) \prod_{j \in \vv \setminus \ww} \omega(x_j)    \rd \bsx_{\uu},  
\end{equation}
for  weight functions $\psi,\omega:\RR \rightarrow (0,\infty)$. The $\|F\|_{\uu, \vv}$ will be essential in estimating localized Walsh coefficients in Lemma~\ref{le:hatF2} and arise naturally from the definition of the Walsh coefficients and integration by parts (see the proof of Lemma~\ref{le:hatF2} for details). Then we define the norm
\begin{equation}\label{Fnorm:infty}
 \|F\|_{\bsgamma,\varphi,\psi,\omega} := \max_{\emptyset \not= \uu \subseteq [s], \vv \subseteq \uu} \frac{1}{\gamma_{\uu}} \|F\|_{\uu, \vv},
\end{equation}
where $[s]:=\{1,\ldots,s\}$, $\bsgamma=\{\gamma_\uu\}_{\uu \subseteq [s]}$ are positive real coordinate-weights and $\psi,\omega$ are weight functions $\RR \rightarrow (0,\infty)$. In Section~\ref{sec:Appl}, where we apply our results, we will choose $\psi=\omega$.  Occasionally, for $\ww \subseteq \uu \subseteq [s]$, $\uu \not=\emptyset$, we may write the derivatives as $$\frac{\partial^{|\uu| + |\ww|} F(\bsx)}{\partial \bsx_\ww \partial \bsx_{\uu}}=\frac{\partial^{|(\bsone_{\uu \setminus \ww},\bstwo_{\ww},\bszero)|}}{\partial \bsx_{(\bsone_{\uu \setminus \ww},\bstwo_{\ww},\bszero)}} F(\bsx)=F^{(\bsone_{\uu \setminus \ww},\bstwo_{\ww},\bszero)}(\bsx)$$ meaning that we take the one-fold partial derivative with respect to $x_j$ when $j \in \uu \setminus \ww$ and the two-fold partial derivative with respect to $x_j$ when $j \in \ww$. 

In the following, we define the worst-case integration error of integrating a function $F$ with respect to the probability density function $\varphi$. 

\begin{definition}
We consider functions $F$ in a function space whose norm is weighted by the functions $\psi$ and $\omega$, see \eqref{eq_uvnorm}, \eqref{Fnorm:infty}. Let $\mathcal{F}_{\bsgamma,\varphi,\psi,\omega} := \{F \in L_2(\mathbb{R}^s, \varphi): \|F\|_{\bsgamma, \varphi,\psi, \omega} < \infty\}$. 
Our analysis focuses on the absolute integration error $|\mathrm{err}(F; \varphi; \mathcal{P})|$ for functions $F$ that have a finite norm. 
In particular, we study the worst-case error, defined by 
\begin{equation}\label{def:wce}
{\rm wce}_{\bsgamma,\varphi,\psi,\omega}(\cP) := \sup_{F \in \mathcal{F}_{\bsgamma,\varphi,\psi,\omega} \atop \|F\|_{\bsgamma,\varphi,\psi,\omega}\le 1} |\mathrm{err}(F; \varphi; \mathcal{P})|.
\end{equation}
\end{definition}

\section{Localized Walsh functions}\label{sec:lwf}

As a technical auxiliary tool, we introduce the concept of localized Walsh functions on the bounded interval $[0,1)$, designed to combine the properties of standard Walsh functions in base $2$ and Haar functions. We then extend the concept of localized Walsh functions to probability density functions with unbounded support via their CDFs. 

\begin{definition}[Walsh functions] Let $k = \kappa_0 + \kappa_1 2 + \cdots + \kappa_{m-1} 2^{m-1}$ be the dyadic expansion of $k \in \mathbb{N}_0$ with digits $\kappa_0, \kappa_1, \ldots, \kappa_{m-1} \in \{0,1\}$. Similarly, let $x \in [0,1)$ have the dyadic expansion $x = \xi_1 2^{-1} + \xi_2 2^{-2} + \cdots$ with $\xi_1, \xi_2, \ldots \in \{0, 1\}$, assuming infinitely many of the digits $\xi_i$ differ from $1$. For $k \in \mathbb{N}_0$ we define the $k$-th Walsh function $\mathrm{wal}_k: [0,1) \to \{-1,1\}$ by
\begin{equation*}
\mathrm{wal}_k(x) := (-1)^{\xi_1 \kappa_0 + \xi_2 \kappa_1 + \cdots + \xi_{m} \kappa_{m-1}}.
\end{equation*}
\end{definition}

Details on Walsh functions can be found in \cite{Fine} and \cite{SWS}. In particular, we use the following lemma, whose proof can be found in Appendix~\ref{app:A}:
\begin{lemma}\label{leWal}
For $r \in \NN_0$, $k \in \{2^r,\ldots,2^{r+1}-1\}$ and $x \in [0,1]$ we have
$$\int_0^x \wal_k(t) \rd t = \frac{1}{2^r}\, \wal_k(x) \, \eta (2^r x),$$ 
where $\eta:\RR \rightarrow \RR$ is the one-periodic function defined by $$\eta (x) := 
\left \{  
\begin{array}{ll} 
x & \mbox{if $0 \le x < \frac{1}{2}$,} \\
x-1 & \mbox{if $\frac{1}{2} \le x < 1$,}
\end{array}\right.$$
and consequently $$ \sup_{x \in [0,1]}\left|\int_0^x \wal_k(t) \rd t\right| = \frac{1}{2^{r+1}}.$$
\end{lemma}

In QMC theory, Walsh functions have demonstrated their utility in deriving error bounds for numerical integration, especially for functions exhibiting higher degrees of smoothness, see, e.g., \cite{DP10}. However, a weakness of Walsh functions is that they are not localized, which presents challenges in analysing the unbounded transformations $\Phi^{-1}$. To mitigate this limitation, we propose the use of \emph{second-order localized Walsh functions}, which amalgamate the beneficial approximation properties of Walsh functions applicable to twice-differentiable functions with the localized features inherent in Haar functions.


To facilitate the introduction of localized Walsh functions, we define the sets $$N_0:=\{0\} \qquad \mbox{ and }\qquad N_1:=N_0 \cup \{2^{a_1}  :  a_1 \in \NN_0\},$$ consisting of zero and all powers of two, and $$N_2:=N_1 \cup \{2^{a_1}+2^{a_2}:  a_1,a_2 \in \NN_0, a_1 >a_2\}$$ containing zero, all powers of two and all sums of two distinct powers of two. If, for $\alpha \in \{1,2\}$, $k \in N_{\alpha}\setminus N_{\alpha-1}$, then the dyadic expansion of $k$ has exactly $\alpha$ non-zero digits. Letting $\alpha \in \{1,2\}$, we define 
$$L_{k, \alpha} := \begin{cases} \{0\} & \mbox{for } k \in N_{\alpha-1},\\ \{0, 1, \ldots, 2^{a_{\alpha}} -1 \} & \mbox{for } k \in N_\alpha \setminus N_{\alpha-1}.\end{cases}$$ Moreover, for $a \in \NN_0$ and $m \in \{0,1,\ldots,2^a-1\}$ we set $I(m,a):=\left[\frac{m}{2^a},\frac{m+1}{2^a}\right)$. The symbol $\boldsymbol{1}_A$ denotes the indicator function of a set $A$. Now we introduce the concept of order $\alpha$ localized Walsh functions.

\begin{definition}[Localized Walsh functions]\label{def:how}
For $\alpha = 1$ we define the set of \textbf{order $1$ localized Walsh functions} $\{w_{k,m, 1}: m \in L_{k, 1}, k \in N_1\}$ as follows:
\begin{enumerate}
\item For $k \in N_{0}$ we have $k=0$ and we define $w_{0,0,1}:[0,1) \to \mathbb{R}$ as $w_{0,0, 1} := \mathrm{wal}_0 = 1$. 
\item For $k = 2^{a_1} \in N_1 \setminus N_0$ and $m \in L_{k,1}$ we define 
\begin{equation*}
w_{k,m, 1}(x) := 2^{a_{1}/2} \, \boldsymbol{1}_{I(m,a_1)} (x) \mathrm{wal}_{k}(x).
\end{equation*}
\end{enumerate}
For $\alpha = 2$ we define the set of \textbf{order $2$ localized Walsh functions} $\{w_{k,m, 2}: m \in L_{k, 2}, k \in N_2\}$ as follows:
\begin{enumerate}
\item For $k \in N_1$ we define $w_{k,0,2}:[0,1)\to \mathbb{R}$ as $w_{k,0,2} : = \mathrm{wal}_k$. 
\item For $k = 2^{a_1} + 2^{a_2} \in N_2 \setminus N_{1}$ with $a_1 > a_2 \ge 0$ and $m \in L_{k, 2}$ define $w_{k, m, 2}: [0,1) \to\mathbb{R}$ by
\begin{equation*}
w_{k,m, 2}(x) := 2^{a_{2}/2} \, \boldsymbol{1}_{I(m,a_2)} (x) \mathrm{wal}_{k}(x).
\end{equation*}
\end{enumerate}
\end{definition}

As an example, in Figure~\ref{f:wal} the 20-th Walsh function $\wal_{20}$ is plotted over $[0,1]$ and in Figure~\ref{f:w} the localized Walsh functions $w_{20,m,2}$ are plotted for $m \in L_{20,2}$. Note that $20=2^4+2^2 \in N_2\setminus N_1$, i.e., $a_1=4$, $a_2=2$ and hence $L_{20,2}=\{0,1,2,3\}$.   

\begin{figure}[h]
\begin{center}
\includegraphics[width = 0.60\textwidth]{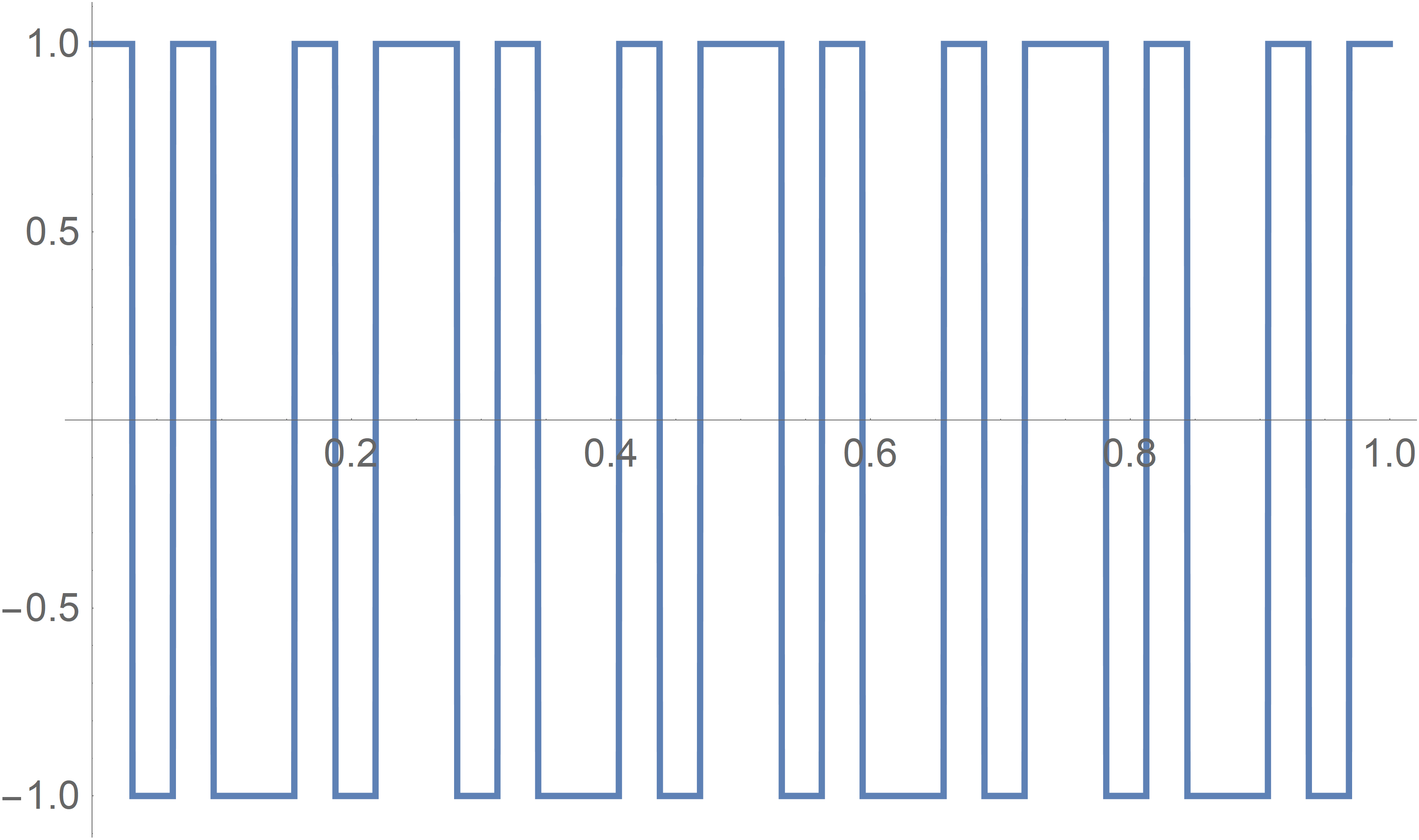}
\caption{Walsh function $\wal_{20}$ }\label{f:wal}
\end{center}
\end{figure}

\begin{figure}[h]
\begin{center}
\includegraphics[width = 0.60\textwidth]{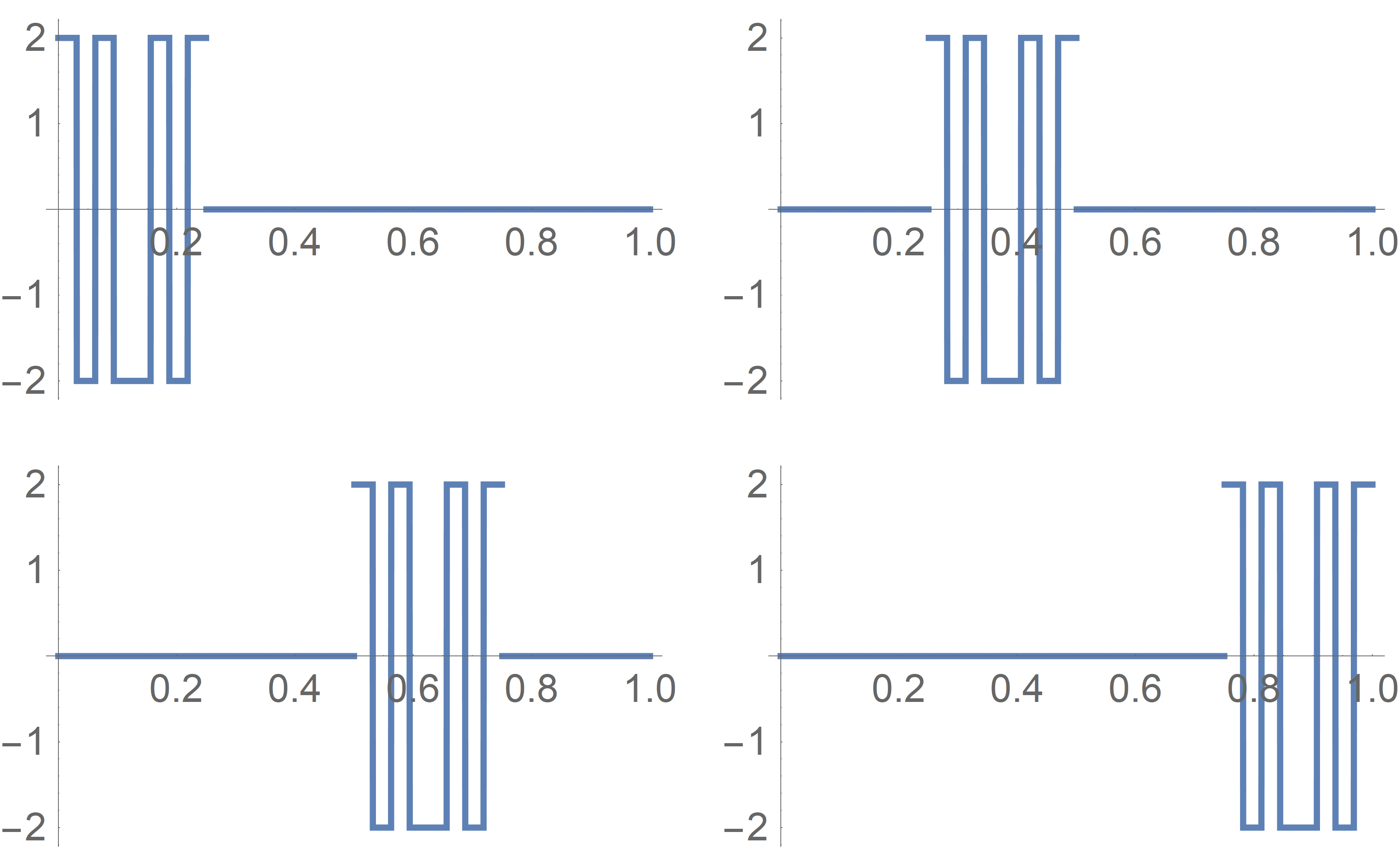}
\caption{Localized Walsh functions $w_{20,m,2}$ for $m \in L_{12,2}=\{0,1,2,3\}$}\label{f:w}
\end{center}
\end{figure}

\begin{remark}
The following two observations are in order:
\begin{enumerate}
\item The functions defined in item (b) of Definition~\ref{def:how} are localized, whereas the functions in item (a) are not. For $\alpha = 1$ there is only a single non-localized function. When $\alpha = 2$, the proportion of non-localized functions becomes negligible in the limit, as indicated by the density $$\lim_{K \to \infty} \frac{\#\{k \in N_1: k \le K\} }{ \#\{k \in N_2: k \le K\} } = 0. $$ Therefore, although not all functions $w_{m,k,\alpha}$ are localized, the vast majority are.

\item Let $k \in N_{\alpha} \setminus N_{\alpha-1}$. Then $k=2^{a_1}$ if $\alpha=1$, and $k = 2^{a_1} + 2^{a_2}$ with $a_1 > a_2 \ge 0$ if $\alpha=2$. For $\ell \in \{0,1,\ldots,2^{a_\alpha}-1\}$, we have the following relation
\begin{equation}\label{Walsh-holWalsh}
\mathrm{wal}_{k+\ell}(x) = \frac{1}{2^{a_\alpha/2}} \mathrm{wal}_{\ell}(x) \sum_{m=0}^{2^{a_\alpha}-1} w_{k,m,\alpha}(x) = \frac{1}{2^{a_\alpha/2}} \sum_{m=0}^{2^{a_\alpha-1}} \mathrm{wal}_{\ell}\left( \frac{m}{2^{a_\alpha}}\right) w_{k,m,\alpha}(x). 
\end{equation}
This follows because for $x \in I(m',a_\alpha)$,
\begin{align*}
\mathrm{wal}_{\ell}(x) \sum_{m=0}^{2^{a_\alpha}-1} w_{k,m,\alpha}(x) = &  \mathrm{wal}_{\ell}(x) w_{k,m',\alpha}(x)\\ 
= & \mathrm{wal}_{\ell}\left( \frac{m'}{2^{a_\alpha}}\right) w_{k,m',\alpha}(x) 
= \sum_{m=0}^{2^{a_\alpha-1}} \mathrm{wal}_{\ell}\left( \frac{m}{2^{a_\alpha}}\right) w_{k,m,\alpha}(x).
\end{align*}
\end{enumerate}
\end{remark}

For any fixed $\alpha \in \{1,2\}$, the set of order $\alpha$ localized Walsh functions
\begin{equation*}
\{w_{k,m,\alpha} \ : \ k \in N_\alpha, m \in L_{k,\alpha} \}
\end{equation*}
forms an orthonormal basis of $L_2([0,1))$, since for all $k, k' \in N_\alpha$ and $m, m' \in L_{k,\alpha}$,
\begin{equation*}
\int_0^1 w_{k,m,\alpha}(x) w_{k', m',\alpha}(x) \,\mathrm{d} x = \begin{cases} 1 & \mbox{if } (k,m) = (k', m'), \\ 0 & \mbox{otherwise}. \end{cases}
\end{equation*}
Moreover, for every $a \in \mathbb{N}_0$ we have the identity
\begin{equation*}
\mathrm{span}\{ w_{k, m,\alpha}: k \in N_\alpha,  \lfloor 2^{a-1} \rfloor \le k < 2^{a}, m \in L_{k,\alpha} \} = \mathrm{span}\{ \mathrm{wal}_k: \lfloor 2^{a-1} \rfloor \le k < 2^{a}\}.
\end{equation*}
In particular, for $\alpha = 1$ we recover the Haar functions
\begin{equation*}
\{w_{\lfloor 2^{j} \rfloor, m, 1} \ : \ j \in \mathbb{N}_{-1}, m \in L_{\lfloor 2^j \rfloor, \alpha}\}.
\end{equation*}
We also note the connection between localized Walsh functions and the tiling of the Walsh-phase plane as discussed in \cite{Thiele}. 

It should also be added that there is a natural extension of order $\alpha$ localized Walsh functions for $\alpha = 3, 4, 5, \ldots$ and basis different from $2$. However, such extensions are not considered in this paper, as we do not use these functions.

\begin{definition}
For $\alpha \in \{1,2\}$, $k \in N_\alpha$ and $m \in L_{k,\alpha}$ define the cumulated functions
\begin{equation*}
W_{k, m,\alpha}(x) := \int_0^x  w_{k, m,\alpha}(y) \rd y \qquad \mbox{and}\qquad 
\widetilde{W}_{k, m,\alpha}(x) := \int_0^x  W_{k, m,\alpha }(y) \rd y.
\end{equation*}
\end{definition}

\begin{lemma}\label{le:supp:W}
For $k \in N_2\setminus N_1$ of the form $k=2^{a_1}+2^{a_2}$, where $a_1>a_2$, and for $m \in L_{k,2}$ the functions $W_{k,m,2}$ and $\widetilde{W}_{k,m,2}$ are zero in $[0, 1) \setminus (\frac{m}{2^{a_2}},\frac{m+1}{2^{a_2}})$.  
\end{lemma}

The proof of this result is deferred to Appendix~\ref{app:B}. Next, we extend the concept of localized Walsh functions, with $\alpha \in \{1,2\}$, to probability density functions with unbounded support by composing them with the corresponding CDFs.

\begin{definition}[Localized Walsh functions with respect to a PDF]
As above, let $\varphi$ be a PDF over $\mathbb{R}$ and let $\Phi(x):=\int_{-\infty}^x \varphi(t) \rd t$ be the corresponding CDF. For $\alpha \in \{1,2\}$ the {\it (order $\alpha$) localized Walsh functions with respect to $\varphi$} are defined as $w_{k,m,\alpha}^{(\varphi)}: \RR \rightarrow \RR$, 
\begin{equation*}
w_{k,m,\alpha}^{(\varphi)}(x) := w_{k,m,\alpha}(\Phi(x)).
\end{equation*}
\end{definition}

For every $\alpha \in \{1, 2\}$, the order $\alpha$ localized Walsh functions $\{w_{k,m,\alpha}^{(\varphi)}: m \in L_{k,\alpha}, k \in N_{\alpha}\}$ with respect to $\varphi$ form an orthonormal system in $L_2(\RR,\varphi)$, because for all $k, k' \in N_\alpha$ and $m, m' \in L_{k,\alpha}$,
\begin{align*}
\int_{\mathbb{R}} w_{k,m,\alpha}^{(\varphi)}(x) w_{k',m',\alpha}^{(\varphi)}(x) \varphi(x) \rd x = & \int_{0}^1 w_{k,m,\alpha}(x) w_{k',m',\alpha}(x) \rd x= \begin{cases} 1 & \mbox{if } (k,m) = (k', m'), \\ 0 & \mbox{otherwise}. \end{cases}
\end{align*}
The corresponding cumulated functions are defined by
\begin{align}\label{cum:Wphi1}
W_{k,m,\alpha}^{(\varphi)}(x)  :=  \int_{-\infty}^x w_{k,m,\alpha}^{(\varphi)}(y) \varphi(y) \rd y = & \int_0^{\Phi(x)} w_{k,m,\alpha}(y) \rd y = W_{k,m,\alpha}(\Phi(x))
\end{align}
and
\begin{align}\label{cum:Wphi2}
\widetilde{W}_{k,m,\alpha}^{(\varphi)}(x)  :=  \int_{-\infty}^x W_{k,m,\alpha}^{(\varphi)}(y) \varphi(y) \rd y = & \int_0^{\Phi(x)} W_{k,m,\alpha}(y) \rd y = \widetilde{W}_{k,m,\alpha}(\Phi(x)).
\end{align}

From Lemma~\ref{le:supp:W} we immediately obtain:

\begin{lemma}\label{le:supp:WW}
For $k \in N_2\setminus N_1$ of the form $k=2^{a_1}+2^{a_2}$, where $a_1>a_2$, and for $m \in L_{k,2}$ the functions $W_{k,m,2}^{(\varphi)}$ and $\widetilde{W}_{k,m,2}^{(\varphi)}$ are zero in $\mathbb{R}\setminus (\Phi^{-1}(\frac{m}{2^{a_2}}),\Phi^{-1}(\frac{m+1}{2^{a_2}}))$.  
\end{lemma}

\section{Walsh series expansions of functions over $\mathbb{R}^s$}\label{sec:serexp}

This section investigates series expansions of functions $F \in L_2(\mathbb{R}^s,\varphi)$ with the application of both Walsh- and localized Walsh bases. The primary objective is to achieve precise estimation of Walsh coefficients under specific conditions, details of which will be provided in Section~\ref{subsec:Walcoeff}. The discussion commences with expansions utilizing localized Walsh functions.

\subsection{Bounds on localized Walsh coefficients}

For a vector $\bsk=(k_1,\ldots,k_s) \in N_2^s$ put $L_{\bsk,2} :=L_{k_1,2} \times \ldots \times L_{k_s,2}$. For $\bsm =(m_1,\ldots,m_s) \in L_{\bsk,2}$ and $\bsx=(x_1,\ldots,x_s)$ put $w_{\bsk,\bsm,2}(\bsx):= \prod_{j=1}^s w_{k_j,m_j,2}(x_j)$ and similarly for $w_{\bsk,\bsm,2}^{(\varphi)}$. For $F \in L_2(\mathbb{R}^s,\varphi)$ we write
\begin{equation*}
F(\bsx) \sim \sum_{\bsk \in N_2^s} \sum_{\bsm \in L_{\bsk,2}} \widehat{F}^{(\varphi)}(\bsk, \bsm) w^{(\varphi)}_{\bsk,\bsm,2}(\bsx),
\end{equation*}
with localized Walsh function coefficients
\begin{equation*}
\widehat{F}^{(\varphi)}(\bsk,\bsm) := \int_{\mathbb{R}^s} F(\bsx) w^{(\varphi)}_{\bsk,\bsm,2}(\bsx) \varphi(\bsx) \rd \bsx.
\end{equation*}
We assume that the weight functions $\varphi, \psi, \omega$ in \eqref{eq_uvnorm} satisfy the following assumption.

\begin{assumption}\label{assumptionA}
Let $\varphi$ be a probability density function on $\mathbb{R}$ used in \eqref{def:int}, and let $\psi, \omega$ be the weight functions in the norm \eqref{eq_uvnorm}, \eqref{Fnorm:infty}. There exist reals $C>0$ and $\beta \in (0,1]$ such that for every $k \in N_2 \setminus N_1$, $k=2^{a_1}+2^{a_2}$, $a_1 > a_2 \ge 0$, and $m \in L_{k,2}$ we have
$$\sup_{x \in \RR} \frac{|\widetilde{W}_{k,m,2}^{(\varphi)}(x)|}{\varphi(x) \psi(x)} \le C \, \frac{ 2^{a_2/2} }{2^{a_1+ \beta a_2}}$$
and
$$\sup_{x \in \RR} \frac{|\widetilde{W}_{k,m,2}^{(\varphi)}(x) \varphi'(x)|}{\varphi^2(x) \omega(x)} \le C\, \frac{ 2^{a_2/2} }{2^{a_1+ \beta a_2}},$$
where $\varphi'$ is the first derivative of $\varphi$ and $\widetilde{W}^{(\varphi)}_{k,m,2}$ is given in \eqref{cum:Wphi2}.
\end{assumption}

Examples of weight functions that satisfy Assumption~\ref{assumptionA} will be discussed in Section~\ref{sec:ex:ass:A}. Notice that the factor $2^{a_2/2}$ arises from the scaling of the localized Walsh functions $w_{k,m,2}$. This factor will cancel out when we switch to Walsh functions in Section~\ref{subsec:Walcoeff}. When we assume in this work that Assumption~\ref{assumptionA} holds, we tacitly mean that this holds for some constant $C>0$ and parameter $\beta \in (0,1]$.

We note that introducing localized Walsh functions—rather than working directly with the standard Walsh expansions—seems indispensable, since we cannot verify an analogue of Assumption~\ref{assumptionA} in a purely global setting for classical Walsh functions.  By localizing the Walsh functions verifying this becomes feasible, see Section~\ref{sec:ex:ass:A} and, in particular, Remark~\ref{rem_w_local}.

The next lemma provides a bound on a certain sum of absolute values of localized Walsh function coefficients of a function $F \in L_2(\mathbb{R}^s,\varphi)$ under the validity of Assumption~\ref{assumptionA}. Here it is essential that the bound is on the \textit{sum over all shifted localized Walsh functions,} since we need this to obtain a bound on the classical Walsh coefficients in Lemma~\ref{est:wal:F}.

\begin{lemma}\label{le:hatF2}
For $\emptyset \not= \uu \subseteq [s]$ let $(\bsk_\uu,\bszero) :=(y_1,y_2,\ldots,y_s) \in N_2^s$, where $y_j=k_j \in N_2 \setminus \{0\}$ for $j \in \uu$ and $y_j=0$ for $j \in [s] \setminus \uu$. Further let $\vv \subseteq \uu$ be such that for $j \in \vv$ we have $k_j =2^{a_{j,1}}+2^{a_{j,2}} \in N_2 \setminus N_1$, $a_{j,1} > a_{j,2} \ge 0$ and for $j \in \uu \setminus \vv$ we have $k_j =2^{a_{j,1}}\in N_1 \setminus N_0$, $a_{j,1}\ge 0$. Then, under Assumption~\ref{assumptionA}, and with $C_1:=\max(C,\tfrac{1}{2})$, we have 
\begin{equation*}
\sum_{\bsm_{\vv} \in L_{\bsk_{\vv}, 2}} | \widehat{F}^{(\varphi)}((\bsk_\uu,\bszero),(\bsm_\vv,\bszero))| \le C_1^{|\uu|} \, \|F\|_{\uu, \vv}  \prod_{j \in \vv} \frac{2^{a_{j,2}/2}}{2^{a_{j,1}+ \beta a_{j,2}}}  \prod_{j \in \uu \setminus \vv} \frac{1}{2^{a_{j,1}}} .
\end{equation*}
\end{lemma}

\begin{proof}
Since $L_{0,2}=\{0\}$, every element from the corresponding set $L_{(\bsk_\uu,\bszero),2}$ has the form $(m_\uu,\bszero)=(z_1,\ldots,z_s)$ where $z_j=m_j \in L_{k_j,2}$ for $j \in \uu$ and $z_j=0$ for $j \in [s]\setminus \uu$. Using integration by parts and introducing the cumulated localized Walsh functions as described in \eqref{cum:Wphi1}, we have
\begin{align*}
\widehat{F}^{(\varphi)}((\bsk_\uu,\bszero),(\bsm_\uu,\bszero)) 
= & \int_{\mathbb{R}^s} F(\bsx) w^{(\varphi)}_{(\bsk_\uu,\bszero),(\bsm_\uu,\bszero),2}(\bsx) \varphi(\bsx) \rd \bsx\\
= & \int_{\mathbb{R}^s} F(\bsx) \prod_{j \in \uu} w^{(\varphi)}_{k_j,m_j,2}(x_j) \prod_{j \in \uu}\varphi(x_j) \prod_{j \in [s] \setminus \uu}\varphi(x_j) \rd \bsx\\
= & (-1)^{|\uu|} \int_{\mathbb{R}^s} \frac{\partial^{|\uu|} F}{\partial \bsx_\uu}(\bsx) \prod_{j \in \uu} W_{k_j,m_j,2}^{(\varphi)}(x_j) \prod_{j \in [s]\setminus \uu} \varphi(x_j) \rd \bsx.
\end{align*}
By definition, $\vv$ consists of all $j \in \uu$ for which $k_j \in N_2 \setminus N_1$. For $j \in \uu \setminus \vv$ we have $k_j \in N_1 \setminus N_0$. Note that for $j \in \uu \setminus \vv$ we have $L_{k_j,2}=\{0\}$ and hence the corresponding $m_j$ equals 0. Using integration by parts a second time for all coordinates $j \in \vv$ and introducing the cumulated localized Walsh functions from \eqref{cum:Wphi2}, we obtain
\begin{eqnarray*}
\lefteqn{\widehat{F}^{(\varphi)}((\bsk_\uu,\bszero),(\bsm_\vv,\bszero))}  \\
& = & (-1)^{|\uu|} \int_{\mathbb{R}^s} \frac{\partial^{|\uu|} F}{\partial \bsx_\uu}(\bsx)  \prod_{j \in \vv} \frac{W_{k_j,m_j,2}^{(\varphi)}(x_j) \varphi(x_j)}{\varphi(x_j)} \prod_{j \in \uu \setminus \vv} W^{(\varphi)}_{k_j,0, 2}(x_j) \prod_{j \in [s]\setminus \uu} \varphi(x_j) \rd \bsx \\
& = & \sum_{\ww \subseteq \vv}  (-1)^{|\uu| + |\ww|}  \int_{\mathbb{R}^s} \frac{\partial^{|\uu| + |\ww|} F}{\partial \bsx_\ww \partial \bsx_{\uu}}(\bsx) \prod_{j \in \ww} \frac{1}{\varphi(x_j)} \prod_{j \in \vv \setminus \ww} \frac{\varphi'(x_j)}{\varphi^2(x_j)} \prod_{j \in \vv} \widetilde{W}_{k_j,m_j,2}^{(\varphi)}(x_j) \\
& & \qquad \qquad \qquad \qquad \times \prod_{j \in \uu \setminus \vv} W^{(\varphi)}_{k_j, 0, 2}(x_j) \prod_{j \in [s]\setminus \uu} \varphi(x_j)  \rd \bsx\\
& = &\sum_{\ww \subseteq \vv}  (-1)^{|\uu| + |\ww|}  \int_{\mathbb{R}^s} \frac{\partial^{|\uu| + |\ww|} F}{\partial \bsx_\ww \partial \bsx_{\uu}}(\bsx) \prod_{j \in \ww} \frac{\widetilde{W}_{k_j,m_j,2}^{(\varphi)}(x_j)}{\varphi(x_j)} \prod_{j \in \vv \setminus \ww} \frac{\widetilde{W}_{k_j,m_j,2}^{(\varphi)}(x_j) \varphi'(x_j)}{\varphi^2(x_j)}  \\
& & \qquad \qquad \qquad \qquad \times \prod_{j \in \uu \setminus \vv} W^{(\varphi)}_{k_j, 0, 2}(x_j) \prod_{j \in [s]\setminus \uu} \varphi(x_j)  \rd \bsx.
\end{eqnarray*}
Taking the absolute value and applying the triangle inequality, we thus have
\begin{eqnarray*}
\lefteqn{|\widehat{F}^{(\varphi)}((\bsk_\uu,\bszero),(\bsm_\vv,\bszero))| }  \\
& \le & \sum_{\ww \subseteq \vv}    \int_{\mathbb{R}^{\uu}} \left| \int_{\mathbb{R}^{s-|\uu|}} \frac{\partial^{|\uu| + |\ww|} F}{\partial \bsx_\ww \partial \bsx_{\uu}}(\bsx)\prod_{j \in [s]\setminus \uu} \varphi(x_j) \rd \bsx_{[s] \setminus \uu} \right| \\
& &\qquad \times  \prod_{j \in \ww} \frac{|\widetilde{W}_{k_j,m_j,2}^{(\varphi)}(x_j)|}{\varphi(x_j)} \prod_{j \in \vv \setminus \ww} \frac{|\widetilde{W}_{k_j,m_j,2}^{(\varphi)}(x_j) \varphi'(x_j)|}{\varphi^2(x_j)}  \prod_{j \in \uu \setminus \vv} |W^{(\varphi)}_{k_j, 0, 2}(x_j)|   \rd \bsx_{\uu}.
\end{eqnarray*}
Now we introduce the weight functions $\psi,\omega:\RR \rightarrow (0,\infty)$ in the following way: we have
\begin{eqnarray*}
\lefteqn{| \widehat{F}^{(\varphi)}((\bsk_\uu,\bszero),(\bsm_\vv,\bszero)) | }  \\
& \le &\sum_{\ww \subseteq \vv}    \int_{\mathbb{R}^{\uu}} \left| \int_{\mathbb{R}^{s-|\uu|}} \frac{\partial^{|\uu| + |\ww|} F}{\partial \bsx_\ww \partial \bsx_{\uu}}(\bsx)\prod_{j \in [s]\setminus \uu} \varphi(x_j) \rd \bsx_{[s] \setminus \uu} \right|  \prod_{j \in \ww} \psi(x_j) \prod_{j \in \vv \setminus \ww} \omega(x_j)     \\
& & \qquad \qquad \times \prod_{j \in \ww} \frac{|\widetilde{W}_{k_j,m_j,2}^{(\varphi)}(x_j)|}{\varphi(x_j) \psi(x_j)} \prod_{j \in \vv \setminus \ww} \frac{|\widetilde{W}_{k_j,m_j,2}^{(\varphi)}(x_j) \varphi'(x_j)|}{\varphi^2(x_j) \omega(x_j)}  \prod_{j \in \uu \setminus \vv} |W^{(\varphi)}_{k_j, 0, 2}(x_j)| \rd \bsx_{\uu}.
\end{eqnarray*}
Taking the sum over all $\bsm_\vv \in L_{\bsk_\vv,2}$ we obtain
\begin{eqnarray*}
\lefteqn{\sum_{\bsm_\vv \in L_{\bsk_\vv,2}} | \widehat{F}^{(\varphi)}((\bsk_\uu,\bszero),(\bsm_\vv,\bszero))| } \nonumber \\
& \le & \sum_{\ww \subseteq \vv}    \int_{\mathbb{R}^{\uu}} \left| \int_{\mathbb{R}^{s-|\uu|}} \frac{\partial^{|\uu| + |\ww|} F}{\partial \bsx_\ww \partial \bsx_{\uu}}(\bsx)\prod_{j \in [s]\setminus \uu} \varphi(x_j) \rd \bsx_{[s] \setminus \uu} \right|  \prod_{j \in \ww} \psi(x_j) \prod_{j \in \vv \setminus \ww} \omega(x_j)   \nonumber \\
& & \quad \times \sum_{\bsm_\vv \in L_{\bsk_\vv,2}} \left(\prod_{j \in \ww} \frac{|\widetilde{W}_{k_j,m_j,2}^{(\varphi)}(x_j)|}{\varphi(x_j) \psi(x_j)} \prod_{j \in \vv \setminus \ww} \frac{2 |\widetilde{W}_{k_j,m_j,2}^{(\varphi)}(x_j) \varphi'(x_j)|}{\varphi^2(x_j) \omega(x_j)}  \prod_{j \in \uu \setminus \vv} |W^{(\varphi)}_{k_j, 0, 2}(x_j)|\right) \rd \bsx_{\uu}.
\end{eqnarray*}
Remember that $j \in \uu \setminus \vv$ means that $k_j \in N_1\setminus N_0$. In this case we have $k_j=2^{a_{j,1}}$ and $w_{k_j,0,2}=\wal_{2^{a_{j,1}}}$. Thus,
$$\sup_{x \in \RR}|W^{(\varphi)}_{k_j, 0, 2}(x)| =\sup_{x \in \RR}\left|\int_0^{\Phi(x)}\wal_{2^{a_{j,1}}}(y) \rd y \right|=\sup_{t \in [0,1]}\left|\int_0^t \wal_{2^{a_{j,1}}}(y) \rd y \right|=\frac{1}{2^{a_{j,1}+1}},$$ where for the last identity we used Lemma~\ref{leWal}. Applying this, Lemma~\ref{le:supp:WW}, and Assumption~\ref{assumptionA}, we obtain
\begin{eqnarray*}
\lefteqn{\sum_{\bsm_{\vv} \in L_{\bsk_{\vv}, 2}} | \widehat{F}^{(\varphi)}((\bsk_\uu,\bszero),(\bsm_\vv,\bszero))|}  \nonumber  \\
& \le &  \sum_{\ww \subseteq \vv} \prod_{j \in \ww} \frac{C\, 2^{a_{j,2}/2}}{2^{a_{j,1}+ \beta a_{j,2}}} \prod_{j \in \vv \setminus \ww}\frac{C\, 2^{a_{j,2}/2}}{2^{a_{j,1}+ \beta a_{j,2}}} \prod_{j \in \uu \setminus \vv} \frac{1}{2^{a_{j,1}+1}} \nonumber \\
& & \times \int_{\mathbb{R}^{\uu}} \left| \int_{\mathbb{R}^{s-|\uu|}} \frac{\partial^{|\uu| + |\ww|} F}{\partial \bsx_\ww \partial \bsx_{\uu}}(\bsx)\prod_{j \in [s]\setminus \uu} \varphi(x_j) \rd \bsx_{[s] \setminus \uu} \right|   \prod_{j \in \ww} \psi(x_j) \prod_{j \in \vv \setminus \ww} \omega(x_j)  \rd \bsx_{\uu} \nonumber  \\ 
& \le & \max(C,\tfrac{1}{2})^{|\uu|}\, \|F\|_{\uu, \vv} \prod_{j \in \vv} \frac{2^{a_{j,2}/2}}{2^{a_{j,1}+ \beta a_{j,2}}}  \prod_{j \in \uu \setminus \vv} \frac{1}{2^{a_{j,1}}},
\end{eqnarray*}
as claimed.
\end{proof}

\subsection{Bounds on Walsh coefficients}\label{subsec:Walcoeff}

For the error analysis we use the classical Walsh-expansion of functions $F \in L_2(\RR^s,\varphi)$, $$F(\bsx) \sim \sum_{\bsk \in \NN_0^s} \widetilde{F}^{(\varphi)}(\bsk) \wal_{\bsk}^{(\varphi)}(\bsx),\qquad \mbox{for } \bsx \in \RR^s,$$ with Walsh coefficients $$\widetilde{F}^{(\varphi)}(\bsk):=\int_{\RR^s}F(\bsx) \wal_{\bsk}^{(\varphi)}(\bsx) \varphi(\bsx)\rd \bsx,$$ where $\wal_{\bsk}^{(\varphi)}(\bsx):=\wal_{\bsk}(\Phi(\bsx))$. Now we show how to use Lemma~\ref{le:hatF2} in order to estimate Walsh coefficients. 

\begin{lemma}\label{est:wal:F}
For $\vv \subseteq \uu \subseteq [s]$ let $(\bsk_\vv+\bsell_\vv,\bsk_{\uu \setminus \vv},\bszero):=(y_1,\ldots,y_s)\in \NN_0^s$, where 
\begin{itemize}
    \item for $j \in \vv$: $y_j=k_j+\ell_j$ with $k_j=2^{a_{j,1}}+2^{a_{j,2}} \in N_2 \setminus N_1$ and $\ell_j \in \{0,1,\ldots,2^{a_{j,2}}-1\}$,
    \item for $j \in \uu \setminus \vv$: $y_j=k_j$ with $k_j = 2^{a_{j,1}} \in N_1\setminus N_0$, and 
    \item for $j \in [s]\setminus \uu$: $y_j=0$.
\end{itemize}
Then, under Assumption~\ref{assumptionA},  we have
\begin{equation*}
|\widetilde{F}^{(\varphi)}((\bsk_\vv+\bsell_\vv,\bsk_{\uu \setminus \vv},\bszero))| \le  C_1^{|\uu|} \, \|F\|_{\uu, \vv} \, \prod_{j \in \vv} \frac{1}{2^{a_{j,1} + \beta a_{j,2}}} \prod_{j \in \uu \setminus \vv} \frac{1}{2^{a_{j,1}}}.
\end{equation*}
\end{lemma}

\begin{proof}
We have 
\begin{eqnarray*}
\widetilde{F}^{(\varphi)}((\bsk_\vv+\bsell_\vv,\bsk_{\uu \setminus \vv},\bszero))  & = & \int_{\RR^s} F(\bsx) \wal_{(\bsk_\vv+\bsell_\vv,\bsk_{\uu \setminus \vv},\bszero)}^{(\varphi)}(\bsx) \varphi(\bsx) \rd \bsx  \\
& = & \int_{\RR^s} F(\bsx) \wal_{\bsk_\vv+\bsell_\vv}^{(\varphi)}(\bsx_\vv) \wal_{\bsk_{\uu \setminus \vv}}^{(\varphi)}(\bsx_{\uu \setminus \vv}) \varphi(\bsx) \rd \bsx. 
\end{eqnarray*}
According to \eqref{Walsh-holWalsh} we obtain $$\wal_{\bsk_\vv+\bsell_\vv}^{(\varphi)}(\bsx_\vv) = \frac{1}{2^{\sum_{j \in \vv} a_{j,2}/2}} \wal_{\bsell_{\vv}}^{(\varphi)}(\bsx_\vv) \sum_{\bsm_{\vv} \in L_{\bsk_{\vv}, 2}} w_{\bsk_\vv,\bsm_\vv,2}^{(\varphi)}(\bsx_\vv)$$
and, by definition, $\wal_{\bsk_{\uu \setminus \vv}}^{(\varphi)}(\bsx_{\uu \setminus \vv})=w^{(\varphi)}_{\bsk_{\uu \setminus \vv},0,2}(\bsx_{\uu \setminus \vv})$. 
This leads to
\begin{eqnarray*}
\lefteqn{\widetilde{F}^{(\varphi)}((\bsk_\vv+\bsell_\vv,\bsk_{\uu \setminus \vv},\bszero))}\\
& = & \frac{1}{2^{\sum_{j \in \vv} a_{j,2}/2}}  \sum_{\bsm_{\vv} \in L_{\bsk_{\vv}, 2}} \int_{\RR^s} F(\bsx) \wal_{\bsell_{\vv}}^{(\varphi)}(\bsx_\vv) w_{\bsk_\vv,\bsm_\vv,2}^{(\varphi)}(\bsx_\vv) w^{(\varphi)}_{\bsk_{\uu \setminus \vv},0,2}(\bsx_{\uu \setminus \vv}) \varphi(\bsx) \rd \bsx.     
\end{eqnarray*}
For $j \in \vv$ consider
$$\wal_{\ell_j}^{(\varphi)}(x_j) w_{k_j,m_j,2}^{(\varphi)}(x_j) = \wal_{\ell_j}(\Phi(x_j)) w_{k_j,m_j,2}(\Phi(x_j)) = \wal_{\ell_j}(y_j) w_{k_j,m_j,2}(y_j),$$ where we use the shorthand notation $y_j:=\Phi(x_j)$. The support of $w_{k_j,m_j,2}$ is the interval $I(m_j,a_{j,2})=[\frac{m_j}{2^{a_{j,2}}},\frac{m_j+1}{2^{a_{j,2}}})$. Since $\ell_j \in \{0,1,\ldots,2^{a_{j,2}}-1\}$ it follows that $\wal_{\ell_j}(y_j)$ is constant on this interval with value $\wal_{\ell_j}(m_j/2^{a_{j,2}})$. Hence $$\wal_{\ell_j}^{(\varphi)}(x_j) w_{k_j,m_j,2}^{(\varphi)}(x_j) = \wal_{\ell_j}\left(\frac{m_j}{2^{a_{j,2}}}\right)  w_{k_j,m_j,2}^{(\varphi)}(x_j).$$ This gives
\begin{eqnarray*}
\lefteqn{\widetilde{F}^{(\varphi)}((\bsk_\vv+\bsell_\vv,\bsk_{\uu \setminus \vv},\bszero))}\\
& = & \frac{1}{2^{\sum_{j \in \vv} a_{j,2}/2}}  \sum_{\bsm_{\vv} \in L_{\bsk_{\vv}, 2}} \left(\prod_{j \in \vv}\wal_{\ell_j}\left(\frac{m_j}{2^{a_{j,2}}}\right)\right) \int_{\RR^s} F(\bsx) w_{\bsk_\vv,\bsm_\vv,2}^{(\varphi)}(\bsx_\vv) w^{(\varphi)}_{\bsk_{\uu \setminus \vv},0,2}(\bsx_{\uu \setminus \vv}) \varphi(\bsx) \rd \bsx \\
& = & \frac{1}{2^{\sum_{j \in \vv} a_{j,2}/2}}  \sum_{\bsm_{\vv} \in L_{\bsk_{\vv}, 2}} \left(\prod_{j \in \vv}\wal_{\ell_j}\left(\frac{m_j}{2^{a_{j,2}}}\right)\right) \widehat{F}^{(\varphi)}((\bsk_\uu,\bszero),(\bsm_\vv,\bszero)).     
\end{eqnarray*}
Taking the absolute value, applying the triangle inequality and using Lemma~\ref{le:hatF2} we obtain 
\begin{eqnarray*}
|\widetilde{F}^{(\varphi)}((\bsk_\vv+\bsell_\vv,\bsk_{\uu \setminus \vv},\bszero))| & \le & \frac{1}{2^{\sum_{j \in \vv} a_{j,2}/2}}  \sum_{\bsm_{\vv} \in L_{\bsk_{\vv}, 2}} |\widehat{F}^{(\varphi)}((\bsk_\uu,\bszero),(\bsm_\vv,\bszero))|\\
& \le  & C_1^{|\uu|}\, \|F\|_{\uu, \vv} \prod_{j \in \vv} \frac{1}{2^{a_{j,1}+ \beta a_{j,2}}} \prod_{j \in \uu \setminus \vv} \frac{1}{2^{a_{j,1}}},  
\end{eqnarray*}
as claimed.
\end{proof}

\section{Numerical integration using polynomial lattice rules}\label{sec:intPLPS}

In the present section, we return to the main aim of this work, the numerical $\varphi$-weighted integration of functions $F$ over $\mathbb{R}^s$ by means of quasi-Monte Carlo rules of the form \eqref{QMC:rule}. Throughout this section we assume the validity of Assumption~\ref{assumptionA}.

\subsection{Error analysis for arbitrary underlying nodes}

Let $\cP=\{\bsx_0,\bsx_1,\ldots,\bsx_{2^m-1}\}$ be a $2^m$-element point set in $(0,1)^s$ and approximate the $\varphi$-weighted integral $\int_{\RR^s}F(\bsx) \varphi(\bsx) \rd \bsx$ with the QMC-rule \eqref{QMC:rule} that is based on this point set. For the error of this approximation we have
\begin{align*}
{\rm err}(F;\varphi;\cP) = & \frac{1}{2^m} \sum_{n=0}^{2^m-1} F(\Phi^{-1}(\bsx_n))-\int_{\RR^s}F(\bsx) \varphi(\bsx) \rd \bsx    \\
= & \sum_{\bsk \in \NN_0^s\setminus\{\bszero\}} \widetilde{F}^{(\varphi)}(\bsk) \frac{1}{2^m} \sum_{n=0}^{2^m-1} \wal_{\bsk}^{(\varphi)}(\Phi^{-1}(\bsx_n))\\
= & \sum_{\emptyset \not= \uu \subseteq [s]} \sum_{\bsk_\uu \in \NN^{|\uu|}} \widetilde{F}^{(\varphi)}((\bsk_\uu,\bszero)) \frac{1}{2^m} \sum_{n=0}^{2^m-1} \wal_{(\bsk_\uu,\bszero)}(\bsx_n).
\end{align*}
Now we distinguish the entries of $\bsk_\uu \in \NN^{|\uu|}$ which are perfect powers of 2 from those which are not. Let $\vv \subseteq \uu$ be the set of coordinates $j$, where $k_j$ has at least two digits in its binary expansion. Accordingly, $\uu \setminus \vv$ is the set of coordinates $j$ which are perfect powers of two. With this notation we have
\begin{align*}
{\rm err}(F;\varphi;\cP) = & \sum_{\emptyset \not= \uu \subseteq [s]} \sum_{\vv \subseteq \uu} \sum_{\bsk_\vv \in (N_2 \setminus N_1)^{|\vv|}} \sum_{\bsell_\vv \in \bigotimes_{j\in \vv} L_{k_j,2}} \sum_{\bsk_{\uu \setminus \vv} \in (N_1 \setminus N_0)^{|\uu|-|\vv|}} \widetilde{F}^{(\varphi)}((\bsk_\vv+\bsell_\vv,\bsk_{\uu \setminus \vv},\bszero)) \\
& \times \frac{1}{2^m} \sum_{n=0}^{2^m-1} \wal_{((\bsk_\vv+\bsell_\vv,\bsk_{\uu \setminus \vv},\bszero))}(\bsx_n).
\end{align*}
Now we take the absolute value and apply the triangle inequality and Lemma~\ref{est:wal:F} to obtain
\begin{align*}
|{\rm err}(F;\varphi;\cP)| \le & \sum_{\emptyset \not= \uu \subseteq [s]} \sum_{\vv \subseteq \uu} C_1^{|\uu|} \, \|F\|_{\uu, \vv}\\
& \times \sum_{\bsk_\vv \in (N_2 \setminus N_1)^{|\vv|}} \sum_{\bsell_\vv \in \bigotimes_{j\in \vv} L_{k_j,2}} \sum_{\bsk_{\uu \setminus \vv} \in (N_1 \setminus N_0)^{|\uu|-|\vv|}}
\left[\prod_{j \in \vv}  \frac{1}{2^{a_{j,1} + \beta a_{j,2}}}\right] \left[ \prod_{j \in \uu \setminus \vv} \frac{1}{2^{a_{j,1}}}\right] \\
& \times \left|\frac{1}{2^m} \sum_{n=0}^{2^m-1} \wal_{((\bsk_\vv+\bsell_\vv,\bsk_{\uu \setminus \vv},\bszero))}(\bsx_n)\right|,
\end{align*}
where for $j \in \vv$ we denote as usual $k_j =2^{a_{j,1}}+2^{a_{j,2}}$ and for $i \in \uu \setminus \vv$ we use $k_i=2^{a_{i,1}}$.
Introducing positive weights $\gamma_\uu$ we obtain
\begin{align*}
|{\rm err}(F;\varphi;\cP)| \le & \left(\max_{\emptyset \not= \uu \subseteq [s]\atop \vv \subseteq \uu} \frac{1}{\gamma_{\uu}}\|F\|_{\uu, \vv}\right) \mathcal{E}_{\bsgamma,\beta}(\cP) = \|F\|_{\bsgamma,\varphi,\psi,\omega} \, \mathcal{E}_{\bsgamma,\beta}(\cP),
\end{align*} 
where
\begin{align*}
\mathcal{E}_{\bsgamma,\beta}(\cP) := & \sum_{\emptyset \not= \uu \subseteq [s]} \gamma_{\uu} \sum_{\vv \subseteq \uu} C_1^{|\uu|} \\
& \times \sum_{\bsk_\vv \in (N_2 \setminus N_1)^{|\vv|}} \sum_{\bsell_\vv \in \bigotimes_{j\in \vv} L_{k_j,2}} \sum_{\bsk_{\uu \setminus \vv} \in (N_1 \setminus N_0)^{|\uu|-|\vv|}}
\left[\prod_{j \in \vv}  \frac{1}{2^{a_{j,1} + \beta a_{j,2}}}\right] \left[ \prod_{i \in \uu \setminus \vv} \frac{1}{2^{a_{i,1}}}\right] \\
& \times \left|\frac{1}{2^m} \sum_{n=0}^{2^m-1} \wal_{((\bsk_\vv+\bsell_\vv,\bsk_{\uu \setminus \vv},\bszero))}(\bsx_n)\right|.    
\end{align*}

We introduce a weight function on both integers and integer vectors that yield a convenient representation of $\mathcal{E}_{\bsgamma,\beta}(\cP)$. For $k = 2^{a_1} + 2^{a_2} + \ell$ with $a_1 > a_2 \ge 0$ and $\ell \in \{0,1,\ldots, 2^{a_2}-1\}$ let
\begin{equation*}
\mu_{\beta}(k) := a_1 + \beta a_2.
\end{equation*}
For $k = 2^{a_1}$ with $a_1 \ge 0$ let
\begin{equation*}
\mu_{\beta}(k) := a_1.
\end{equation*}
Further, let $\mu_{\beta}(0) := 0$. For $\bsk \in \mathbb{N}_0$ let
\begin{equation*}
\mu_{\beta}(\bsk) := \mu_{\beta}(k_1) + \cdots + \mu_{\beta}(k_s).
\end{equation*}
Then we have
\begin{align}\label{def:calE}
\mathcal{E}_{\bsgamma,\beta}(\cP) =  \sum_{\emptyset \not= \uu \subseteq [s]} \gamma_{\uu} \, (2 \,C_1)^{|\uu|} \sum_{\bsk_\uu \in \mathbb{N}^{|\uu|}} \frac{1}{2^{\mu_{\beta}(\bsk_{\uu})}}  \left|\frac{1}{2^m} \sum_{n=0}^{2^m-1} \wal_{(\bsk_\uu,\bszero)}(\bsx_n)\right|.    
\end{align}

We have now established the following proposition.
\begin{proposition}\label{prob:errbd}
Let $F \in\mathcal{F}_{\bsgamma,\varphi,\psi,\omega}$ and assume that $\psi$ and $\omega$ satisfy Assumption~\ref{assumptionA}. Let $\cP$ be a $2^m$-element point set in $(0,1)^s$. Then we have $$|{\rm err}(F;\varphi;\cP)| \le \|F\|_{\bsgamma,\varphi,\psi,\omega} \, \mathcal{E}_{\bsgamma,\beta}(\cP)$$ where $\mathcal{E}_{\bsgamma,\beta}(\cP)$ is given by \eqref{def:calE}. In particular, for the worst-case error \eqref{def:wce} we obtain $${\rm wce}_{\bsgamma,\varphi,\psi,\omega}(\cP) \le \mathcal{E}_{\bsgamma,\beta}(\cP).$$
\end{proposition}

\subsection{Interlaced polynomial lattice rules of order two}

In this section we discuss a construction of polynomial lattice rules of order two which can achieve optimal convergence rates of \eqref{def:calE}. We consider polynomial lattice over $\mathbb{Z}_2$, the finite field of order $2$, see \cite{niesiam}. We identify $\mathbb{Z}_2$ with $\{0,1\}$ equipped with arithmetic operations modulo $2$. Furthermore, let $\mathbb{Z}_2[x]$ be the ring of polynomials over $\mathbb{Z}_2$ and let $\mathbb{Z}_2((x^{-1}))$ be the ring of Laurent series over $\mathbb{Z}_2$ whose elements are of the form $L(x)=\sum_{\ell=w}^{\infty} t_{\ell} x^{-\ell}$ for $w \in \mathbb{Z}$ and coefficients $t_{\ell}\in \mathbb{Z}_2$. For $m \in \NN$ let
\begin{equation*}
G_m := \{ q \in \mathbb{Z}_2[x] \setminus \{ 0 \} : \deg(q) < m \}. 
\end{equation*}
For every integer $n \in \{0,1,\ldots,2^m-1\}$, let $n = \eta_0 + \eta_1 2 + \cdots + \eta_{m-1} 2^{m-1}$ be the binary expansion of $n$ with digits $\eta_i$ in $\mathbb{Z}_2$, and associate with $n$ the polynomial
\begin{equation*}
n(x) = \sum_{r=0}^{m-1} \eta_r x^r \in \mathbb{Z}_2[x].
\end{equation*}
Furthermore, we denote by $v_m$ the map from $\mathbb{Z}_2((x^{-1}))$ to the set of dyadic rationals $\{a 2^{-m} : a \in \{0,1,\ldots,2^m-1\}\}$, which form a subset of $[0,1)$, defined by
\begin{equation*}
v_m\left( \sum_{\ell=w}^{\infty} t_{\ell} x^{-\ell} \right) = \sum_{\ell=\max(1,w)}^{m} \frac{t_{\ell}}{2^{\ell}}.
\end{equation*}

\begin{definition}[Polynomial lattice rules] 
For any positive integer $m$, choose an irreducible polynomial $P\in\mathbb{Z}_2[x]$ of degree $m$; we refer to $P$ as the \emph{modulus}. For a prescribed dimension $s\in\mathbb{N}$, select polynomials $q_1,\ldots,q_s\in G_m$ and define
\begin{equation*}
\bsq = \bsq(x) = (q_1(x), \dots, q_s(x)) \in G^s_m.
\end{equation*}
The vector $\bsq$ is called the \emph{generating vector}. Then, a {\it (classical) polynomial lattice rule} with $2^m$ points in $s$ dimensions is a QMC rule that is based on the point set $S_{P,m,s}(\bsq)$ which consists of the elements
\[
\bsy_n = \left( v_m\left( \frac{n(x) q_1(x)}{P(x)} \right), \dots, v_m\left( \frac{n(x) q_s(x)}{P(x)} \right) \right) \in [0,1)^s, \quad n \in \{0,1, \dots, 2^m-1\},
\]
for some given modulus $P$ and generating vector $\bsq \in G^{s}_{m}$. 
\end{definition}

Polynomial lattice rules are an important type of QMC rule and constitute a polynomial variant of classical lattice rules with many useful properties. For a detailed exposition, we refer to Chapter~10 of \cite{DP10}. In our analysis, we will need the following character property, whose proof is a combination of Lemma~4.75 and Lemma~10.6 in \cite{DP10}. 

\begin{lemma}\label{le:charprop}
Let $S_{P,m,s}(\bsq):=\{\bsy_0,\bsy_1,\ldots,\bsy_{2^m-1}\}$ be the quadrature points of a polynomial lattice rule with modulus $P \in \mathbb{Z}_2[x]$, $\deg(P)=m$, and generating vector $\bsq=(q_1,\ldots,q_s) \in G_m^s$. Then, for $\bsk=(k_1,\ldots,k_s) \in \{0,1,\ldots,2^m-1\}^s$  we have $$\sum_{n=0}^{2^m-1} \wal_{\bsk}(\bsy_n)=\left\{ 
\begin{array}{ll}
2^m & \mbox{if $k_1(x) q_1(x)+\cdots+k_s(x) q_s(x) \equiv 0 \pmod{P(x)}$,}\\
0 & \mbox{otherwise}.
\end{array}
\right.
$$     
\end{lemma}

In the following we define interlaced polynomial lattice rules \citep{DKLNS14, G15}, belonging to the family of higher order digital nets, which were first introduced in \cite{dick08}.

\begin{definition}[Interlaced polynomial lattice rules] 
Define the digit interlacing function with interlacing factor 2 by
\begin{equation*}
D : [0,1)^2 \to [0,1), \quad (x_1,x_2) \mapsto \sum_{i=1}^{\infty} \left(\frac{\xi_{1,i}}{2^{2i-1}}+\frac{\xi_{2,i}}{2^{2i}}\right),
\end{equation*}
where $x_j = \xi_{j,1} 2^{-1} + \xi_{j,2} 2^{-2} + \cdots$ with binary digits $\xi_{j,i}$ in $\{0,1\}$ for $j \in \{1,2\}$. We also define such a function for vectors by setting
\begin{equation*}
D : [0,1)^{2 s} \to [0,1)^s, \quad (x_1, \dots, x_{2 s}) \mapsto \left( D(x_1,x_2), \dots, D(x_{2 s-1},x_{2 s}) \right).
\end{equation*}
Then, an {\it interlaced polynomial lattice rule of order 2} with $2^m$ points in $s$ dimensions is a QMC rule using $D(S_{P,m,2 s}(\bsq))$ as quadrature points, for some given modulus $P \in \mathbb{Z}_2[x]$ of degree $m$ and and generating vector $\bsq \in G^{2s}_{m}$.
\end{definition}

Unfortunately, any polynomial lattice rule as well as any interlaced polynomial lattice rule contains the origin as quadrature point. However, we require node sets in the open unit cube $(0,1)^s$; see Proposition~\ref{prob:errbd}. To avoid this problem, observe that the coordinates $x_{n}^{(j)}$, where $j \in [s]$ and $n \in \{0,1,\ldots,2^m-1\}$, of the nodes $\bsx_n:=D(\bsy_n)$ of an interlaced polynomial lattice rule of order 2 with $2^m$ points in dimension $s$ are dyadic rationals of the form $x_n^{(j)}=\xi_{n,1}^{(j)} 2^{-1}+\xi_{n,2}^{(j)} 2^{-2}+\cdots+\xi_{n,2 m}^{(j)} 2^{-2 m}$ with dyadic digits $\xi_{n,i}^{(j)} \in \{0,1\}$. Thus, we simply add the (digital) shift $2^{-2m-1}$, resulting in $$x_{n,i}^{(j)} + \frac{1}{2^{2m+1}}=\frac{\xi_{n,1}^{(j)}}{2}+\frac{\xi_{n,2}^{(j)}}{2^2}+\cdots+\frac{\xi_{n,2 m}^{(j)}}{2^{2 m}}+\frac{1}{2^{2m+1}},$$ to every coordinate of every point of the interlaced polynomial lattice rule. If $\cP_{m,s}:=D(S_{P,m,2 s}(\bsq))$ denotes the interlaced polynomial lattice point set, then we denote by $\cP_{m,s}^{{\rm sh}}$ the {\it shifted interlaced polynomial lattice point set} $\cP_{m,s} + 2^{-2m-1} \boldsymbol{1}$ thereof, where $\boldsymbol{1}$ is the all-1 vector in dimension $s$. Obviously, $\cP_{m,s}^{{\rm sh}}$ belongs to $(0,1)^s$, and thus avoids the origin.

It is sometimes useful to also add a digital shift to the earlier digits. To do so, we need digital addition modulo $2$. For $x = x_1 2^{-1} + x_2 2^{-2} + \cdots$ and $y = y_1 2^{-1} + y_2 2^{-2} + \cdots$, we define $z = x \oplus y$ as $z = z_1 2^{-1} + z_2 2^{-2} + \cdots$ by $z_i := x_i + y_i \pmod{2}$ for $i \in \NN$. For vectors $\bsx, \bsy$ we define $\bsx \oplus \bsy := (x_1 \oplus y_1, \ldots, x_s \oplus y_s)$. Let $\bsdelta \in [0,1)^s$. Then we define the digitally shifted polynomial lattice rule $\cP_{m,s}^{{\rm sh}} := \{\bsx \oplus \bsdelta: \bsx \in \cP_{m,s}\}$. For $\bsdelta \in (0,1)^s$ such that $\bsdelta 2^{2m} \notin \mathbb{Z}$, we obtain that none of the components of $\cP_{m,s}^{{\rm sh}} \oplus \bsdelta$ are $0$ and hence all the components of all points avoid the origin. Note that for $\bsdelta = 2^{-2m-1} \bsone$ we have $\cP_{m,s} + \bsdelta = \cP_{m,s} \oplus \bsdelta$. 

We also extend the definition of the interlacing function $D:[0,1)^2 \rightarrow [0,1)$ to nonnegative integers by setting
\[
E : \mathbb{N}_0^2 \to \mathbb{N}_0, \quad (\ell_1,  \ell_2) \mapsto \sum_{a=0}^{\infty} (\ell_{1,a} 2^{2 a}+\ell_{2,a} 2^{2 a+1}), 
\]
where $\ell_j = \ell_{j,0} + \ell_{j,1} 2 + \ell_{j,2} 2^2 + \cdots$ for $j \in \{1,2\}$. We also extend this function to vectors via
\[
E : \mathbb{N}_0^{2 s} \to \mathbb{N}_0^s, \quad (\ell_1, \dots, \ell_{2 s}) \mapsto \left( E(\ell_1, \ell_2), \dots, E(\ell_{2 s-1}, \ell_{2 s}) \right). 
\]
Note that for $k,\ell \in \NN$ the $j$-th digit of $k$ is the digit at position $2j-1$ in 
$E(k,\ell)$ and the $j$-th digit of $\ell$ is the digit at position $2j$ in $E(k,\ell)$.

\begin{lemma}\label{lem_alpha}
For $k \in \mathbb{N}_0$ let
\begin{equation*}
\mu(k) := \begin{cases} 0 & \mbox{if } k = 0, \\ \lfloor \log_2(k) \rfloor & \mbox{if } k \in \mathbb{N}. \end{cases}
\end{equation*}
Then, for $\beta \in (0,1]$ we have $\mu_{\beta}(E(0,0)) =  0$ and, for $(k_1, k_2) \in \mathbb{N}_0^2 \setminus \{(0, 0)\}$, we have
\begin{equation*}
\mu_{\beta}(E(k_1, k_2)) \ge (1+\beta) \, (\mu(k_1) +  \mu(k_2) - 1).
\end{equation*}
\end{lemma}

\begin{proof}
First, assume that $k_1 = 2^{a_1} + \ell > k_2 = 2^{b_1}  + m > 0$. Then $a_1 \ge b_1$ and
\begin{equation*}
\mu_{\beta}(E(k_1, k_2)) \ge  2a_1 - 1 + \beta 2b_1 \ge (1+\beta) (a_1 + b_1) -1 \ge  (1+\beta)\, (\mu(k_1) + \mu(k_2) - 1).
\end{equation*}
The proof is similar if $k_1 \le k_2$. Now consider $k_2 = 0$. Then $\mu(k_2) = 0$, and for $k_1 = 2^{a_1} + 2^{a_2} + \ell$, we have
\begin{equation*}
\mu_{\beta}(E(k_1, k_2)) = 2a_1-1 + \beta (2a_2-1) \ge  (1+\beta) \mu(k_1) - (1+ \beta),
\end{equation*}
while for $k_1 = 2^{a_1}$ we have
\begin{equation*}
\mu_{\beta}(E(k_1, k_2)) = 2a_1-1 \ge 2 \mu(k_1) - (1 + \beta).
\end{equation*}
The case where $k_1 = 0$ and $k_2 > 0$ is similar. The case where $k_1=k_2 = 0$ is obvious.
\end{proof}

For a given set $\emptyset \neq \vv \subseteq [2 s]$, we define
\[
\uu(\vv) := \{\lceil j/2 \rceil : j \in \vv\} \subseteq [s], 
\]
where each element appears only once as is typical for sets. The set $\uu(\vv)$ can be viewed as an indicator on whether the set $\vv$ includes any element from each block of $2$ components from $[2 s]$. In the following, we will use the fact that $|\uu(\vv)| \le |\vv| \le 2|\uu(\vv)| $.

For a vector $\bsk=(k_1,\ldots,k_d) \in \NN_0^d$ of arbitrary length $d \in \NN$, put $\mu(\bsk):=\sum_{j=1}^d \mu(k_j)$. We now apply Lemma~\ref{lem_alpha} to the vector $(\bsell_{\vv}, \bszero) \in \NN_0^{2s}$ in blocks of $2$ components.
For $\bsell_{\vv} \in \mathbb{N}^{2|\vv|}$, noting that $\mu_{\beta}(E(\bszero)) = 0$, we obtain
\begin{equation}\label{est:mubbmu}
\mu_{\beta}(E(\bsell_\vv, \bszero)) \geq  (1 + \beta) (\mu(\bsell_\vv) - |\uu(\vv)|).
\end{equation}

\subsection{Component-by-component construction}

We summarize the fundamental choice of underlying integration points in the following analysis:
\begin{itemize}
\item Let $m,s\in \mathbb{N}$ and let $S_{P, m, 2s}(\bsq) := \{\bsy_0, \bsy_1, \ldots, \bsy_{2^m-1}\}$ in $[0,1)^{2s}$ be the nodes of a polynomial lattice rule in dimension $2 s$ with modulus $P \in \mathbb{Z}_2[x]$ of degree $m$ and generating vector $\bsq \in G_m^{2s}$. 
\item Let $\cP_{m,s} := D(S_{P,m,2s}(\bsq)) = \{D(\bsy_0), D(\bsy_1), \ldots, D(\bsy_{2^m-1})\}$ in $[0,1)^s$ be the interlaced version of $S_{P, m, 2s}(\bsq)$. 
\item Let $\cP_{m,s}^{{\rm sh}}:=\cP_{m,s} + 2^{-2m-1} \boldsymbol{1}$ or $\cP_{m,s}^{{\rm sh}} := \cP_{m,s} \oplus \bsdelta$ be the shifted version of $\cP_{m,s}$, where $\bsdelta$ is chosen such that $ \cP_{m,s} \subset (0,1)^s$.
\end{itemize}
Then we study the error bound $\mathcal{E}_{\bsgamma,\beta}(\cP_{m,s}^{{\rm sh}})$ from Proposition~\ref{prob:errbd}. We have
\begin{align*}
\mathcal{E}_{\bsgamma,\beta}(\cP_{m,s}^{{\rm sh}})=  \sum_{\emptyset \not= \uu \subseteq [s]} \gamma_{\uu} \, (2 \, C_1)^{|\uu|} \sum_{\bsk_\uu \in \mathbb{N}^{|\uu|}} \frac{1}{2^{\mu_{\beta}(\bsk_{\uu})}}  \left|\frac{1}{2^m} \sum_{n=0}^{2^m-1} \wal_{(\bsk_\uu,\bszero)}(D(\bsy_n) +2^{-2m-1} \boldsymbol{1})\right|.
\end{align*}
Note that $$\wal_{(\bsk_\uu,\bszero)}(D(\bsy_n) \oplus \bsdelta)=\wal_{(\bsk_\uu,\bszero)}(D(\bsy_n)) \wal_{(\bsk_\uu,\bszero)}(\bsdelta)$$ and $|\wal_{(\bsk_\uu,\bszero)}(\bsdelta)|=1$. Hence, applying the absolute value, we get rid of the dependence on the shift. This way we obtain 
\begin{align*}
\mathcal{E}_{\bsgamma,\beta}(\cP_{m,s}^{{\rm sh}}) & = \sum_{\emptyset \not= \vv \subseteq [2s]} \gamma_{\uu(\vv)} \, (2 \,C_1)^{|\uu(\vv)|} \sum_{\bsk_\vv \in \mathbb{N}^{|\vv|}} \frac{1}{2^{\mu_{\beta}(E(\bsk_{\vv}, \bszero))}}  \left|\frac{1}{2^m} \sum_{n=0}^{2^m-1} \wal_{E(\bsk_\vv,\bszero)}(D(\bsy_n))\right|\\
& = \sum_{\emptyset \not= \vv \subseteq [2s]} \gamma_{\uu(\vv)} \, (2 \, C_1)^{|\uu(\vv)|} \sum_{\bsk_\vv \in \mathbb{N}^{|\vv|}} \frac{1}{2^{\mu_{\beta}(E(\bsk_{\vv}, \bszero))}}  \left|\frac{1}{2^m} \sum_{n=0}^{2^m-1} \wal_{(\bsk_\vv,\bszero)}(\bsy_n)\right|,
\end{align*}
where $(\bsk_\vv, \bszero) \in \mathbb{N}_0^{2s}$ and where we used that for a general $\bsk \in \mathbb{N}_0^{2s}$ and $\bsx \in [0,1)^{2s}$ we have $\wal_{E(\bsk)}(D(\bsx)) = \wal_{\bsk}(\bsx)$. Using the estimate \eqref{est:mubbmu} and afterwards Lemma~\ref{le:charprop} we obtain
\begin{align}\label{eq:EBsq}
\mathcal{E}_{\bsgamma,\beta}(\cP_{m,s}^{{\rm sh}})
& \le  \sum_{\emptyset \not= \vv \subseteq [2s]} \gamma_{\uu(\vv)} \, (2 \, C_1)^{|\uu(\vv)|}  2^{(1+\beta) |\uu(\vv)|} \sum_{\bsk_\vv \in \mathbb{N}^{|\vv|}} \frac{1}{2^{(1+\beta) \mu(\bsk_{\vv},\bszero)}}  \left|\frac{1}{2^m} \sum_{n=0}^{2^m-1} \wal_{(\bsk_\vv,\bszero)}(\bsy_n)\right| \nonumber\\  
& = \sum_{\emptyset \not= \vv \subseteq [2s]} \gamma_{\uu(\vv)} \, \left(C_1 \, 2^{2+\beta}\right)^{|\uu(\vv)|} \sum_{\bsk_\vv \in S^\perp_{\vv}(\bsq)} \frac{1}{2^{(1+\beta) \mu(\bsk_{\vv}, \bszero)}}\\
& = : B_s(\bsq),\nonumber
\end{align}
where
\begin{equation*}
S^\perp_{\vv}(\bsq) := \left\{ \bsk_\vv \in \mathbb{N}^{|\vv|}: \sum_{j \in \vv} {\rm tr}_m(k_j)(x) q_j(x) \equiv 0 \pmod{P(x)}\right\},
\end{equation*}
where, for $k=\sum_{i=0}^{\infty} \kappa_i 2^i$ in $\mathbb{N}$ with binary digits $\kappa_i \in \{0,1\}$, which become eventually 0 from a certain index on, the $m$-digit truncation is defined as ${\rm tr}_m(k):=\kappa_0 +\kappa_1 2+\cdots+\kappa_{m-1} 2^{m-1}$, which belongs to $\{0,1,\ldots,2^m-1\}$. Remember that $C_1 = \max(C, \tfrac12)$ and $C > 0$ is the constant from Assumption~\ref{assumptionA}.

Thus, we use $B_s(\bsq)$ as the error criterion for a (shifted) interlaced polynomial lattice rule that dominates the worst-case integration error ${\rm wce}_{\bsgamma,\varphi}(\cP_{m,s}^{{\rm sh}})$, according to Proposition~\ref{prob:errbd} and \eqref{eq:EBsq}. The expression for $B_s(\bsq)$ in \eqref{eq:EBsq} is very similar to Eq.~(3.30) in \cite{DKLNS14} where $b = 2$ and $\alpha = 1+\beta$. The results apply with some minor modifications. In particular, we can apply Theorem~3.9 in \cite{DKLNS14} to obtain the following result.

\begin{lemma}[CBC Error Bound]\label{CBCerrorbound}
Let $m, s \in \mathbb{N}$ and let $P \in \mathbb{Z}_2[x]$ be an irreducible polynomial with $\deg(P) = m $. Let $\bsgamma=\{\gamma_{\uu}\}_{\uu \subseteq [s]}$ be positive real coordinate weights. Then a generating vector $\bsq^* = (1, q_2^*, \ldots, q_{2s}^*) \in G_m^{2s} $ can be constructed using a component-by-component approach, minimizing $B_s(\bsq^\ast)$ in each step, such that for all $ \lambda \in (\frac{1}{1+\beta}, 1]$ we have
\begin{equation*}
B_s(\bsq^\ast) \leq \left( \frac{4}{2^m} \sum_{\emptyset \not= \uu \subseteq [s]} \gamma_{\uu}^\lambda K_{\beta,\lambda}^{|\uu|} \right)^{1/\lambda},
\end{equation*}
where $K_{\beta,\lambda}:=12 (C_1 / (2^{\lambda (1+\beta)} - 2))^2$.
\end{lemma}

\begin{proof}
From Theorem~3.9 in \cite{DKLNS14} we obtain
\begin{equation*}
B_s(\bsq^\ast) \leq \left( \frac{2}{2^m-1} \sum_{\emptyset \not= \vv \subseteq [2s]} \gamma_{\uu(\vv)}^\lambda \frac{ (C_1 \, 2^{2+\beta})^{\lambda|\uu(\vv)|}}{ (2^{(1+\beta) (\lambda+1)} - 2^{2+\beta})^{|\vv|}} \right)^{1/\lambda},
\end{equation*}
for all $ \lambda \in ( \frac{1}{1+\beta}, 1]$. Obviously, $2/(2^m-1)\le 4/2^m$ and furthermore, 
$$
\frac{(C_1 \, 2^{2+\beta})^{\lambda|\uu(\vv)|}}{ (2^{(1+\beta) (\lambda+1)} - 2^{2+\beta})^{|\vv|}} =  \frac{(C_1 \, 2^{2+\beta})^{\lambda|\uu(\vv)|}}{2^{(1+\beta)|\vv|}} \frac{1}{(2^{\lambda (1+\beta)} - 2)^{|\vv|}}
\le J_{\beta,\lambda}^{|\vv|},
$$
with $J_{\beta,\lambda}:=2\, C_1 / (2^{\lambda (1+\beta)} - 2)$. This implies 
\begin{equation*}
B_s(\bsq^\ast) \leq \left( \frac{4}{2^m} \sum_{\emptyset \not= \vv \subseteq [2s]} \gamma_{\uu(\vv)}^\lambda J_{\beta,\lambda}^{|\vv|} \right)^{1/\lambda}.
\end{equation*}
Now, recall that for $\emptyset \not=\vv \subseteq [2 s]$ we have $\uu(\vv)=\{\lceil j/2\rceil \ : \ j \in \vv\} \subseteq [s]$, where each element appears only once, as is typical for sets, in particular $|\vv|/2 \le |\uu(\vv)| \le | \vv|$. Hence
\begin{align*}
\sum_{\emptyset \not= \vv \subseteq [2s]} \gamma_{\uu(\vv)}^\lambda J_{\beta,\lambda}^{|\vv|} \le &  \sum_{\emptyset \not= \uu \subseteq [s]} \gamma_{\uu}^\lambda J_{\beta,\lambda}^{2 |\uu|} \sum_{\vv \subseteq [2s] \atop \uu(\vv)=\uu} 1 \le \sum_{\emptyset \not= \uu \subseteq [s]} \gamma_{\uu}^\lambda (3 J_{\beta,\lambda}^2)^{|\uu|},  
\end{align*}
where we used that $\sum_{\vv \subseteq [2s] \atop \uu(\vv)=\uu} 1 \le 3^{|\uu|}$. Setting $K_{\beta,\lambda}:=3 J_{\beta,\lambda}^2$ we obtain the claimed result.
\end{proof}

Combining Lemma~\ref{CBCerrorbound} with Proposition~\ref{prob:errbd} and Eq.~\eqref{eq:EBsq} we obtain the following bound on the worst-case error~\eqref{def:wce}. 

\begin{theorem}\label{thm:error_bound}
Consider the approximation of the integral $\int_{\mathbb{R}^s} F(\bsx) \varphi(\bsx) \,\mathrm{d} \bsx$ of functions $F \in \mathcal{F}_{\bsgamma,\varphi,\psi,\omega}$ and assume that the weight functions $\psi,\omega$ used in \eqref{eq_uvnorm}, \eqref{Fnorm:infty} to define the norm in $\mathcal{F}_{\bsgamma,\varphi,\psi,\omega}$, satisfy Assumption~\ref{assumptionA} (which determines $\beta$). Let $m, s \in \mathbb{N}$ and let $P \in \mathbb{Z}_2[x]$ be an irreducible polynomial with $\deg(P) = m $. Let $\bsgamma=\{\gamma_{\uu}\}_{\uu \subseteq [s]}$ be positive real coordinate weights. Then a generating vector $\bsq^* = (1, q_2^*, \ldots, q_{2s}^*) \in G_m^{2s}$ can be constructed using a component-by-component approach, minimizing $ B_s(\bsq)$ given in \eqref{eq:EBsq} in each step, such that for the shifted interlaced lattice rule of order 2 generated by $\bsq^*$ with underlying quadrature points $\cP_{m,s}^{{\rm sh}}$, such that for all $ \lambda \in ( \frac{1}{1+\beta}, 1]$ the worst-case integration error satisfies
\begin{equation*}
{\rm wce}_{\bsgamma,\varphi,\psi,\omega}(\cP_{m,s}^{{\rm sh}})\le  \left( \frac{4}{2^m} \sum_{\emptyset \not= \uu \subseteq [s]} \gamma_{\uu}^\lambda K_{\beta,\lambda}^{|\uu|} \right)^{1/\lambda},
\end{equation*}
where $K_{\beta,\lambda}:=12 (C_1 / (2^{\lambda (1+\beta)} - 2))^2$, as in Lemma~\ref{CBCerrorbound}.
\end{theorem}

The convergence rate is $N^{-1/\lambda}$, where $N=2^m$, with $\lambda$ arbitrarily close to $1/(1+\beta)$. In the example in Section~\ref{sec:ex:ass:A} we consider integration with respect to $\varphi=\varphi_{\sigma}$, the normal PDF with mean 0 and variance $\sigma^2$, and we choose $\psi=\omega=\varphi_\rho$ with $\rho\ge \sigma$. In this case we will be able to choose $\beta=1-\sigma^2/\rho^2-\delta$ where $\delta \in (0,1-\sigma^2/\rho^2)$ can be chosen arbitrarily close to 0 (see Proposition~\ref{pr:beta}). Thus, $1+\beta=2-\sigma^2/\rho^2-\delta$. In this case we obtain a convergence rate of order $N^{-2+\sigma^2/\rho^2+\delta}$ where $\delta$ can be arbitrarily close to~0. A further example is listed in Section~\ref{sec:ex:ass:A} (see Proposition~\ref{pr:beta3}).

If, for some $\lambda \in (\frac{1}{1+\beta},1]$, we have 
\begin{equation}\label{cond:w:spt}
\Gamma_{\beta,\lambda}:=\sup_{s \in \NN} \sum_{\emptyset \not= \uu \subseteq [s]} \gamma_{\uu}^\lambda K_{\beta,\lambda}^{|\uu|}< \infty,
\end{equation}
then we have 
$${\rm wce}_{\infty,\bsgamma,\varphi,\psi,\omega}(\cP_{m,s}^{{\rm sh}})\le  \left( \frac{4 \Gamma_{\beta,\lambda}}{2^m}\right)^{1/\lambda},$$ 
and this bound holds uniformly in $s$. In particular, for $\varepsilon \in (0,1)$ and $s \in \NN$ define the so-called information-complexity 
\[
N(\varepsilon,s):=\min\{N \in \NN \ : \ \exists\, \cP_N \subseteq (0,1)^s \mbox{ s.t. $|\cP_N|=N$ and } {\rm wce}_{\infty,\bsgamma,\varphi,\psi,\omega}(\cP_N) \le \varepsilon\}.
\] 
Then, for $m \in \NN$ such that $2^{m-1} < 4 \Gamma_{\beta,\lambda} \varepsilon^{-\lambda} \le 2^m$ we have ${\rm wce}_{\infty,\bsgamma,\varphi,\psi,\omega}(\cP_{m,s}^{{\rm sh}}) \le \varepsilon$ and hence 
$$N(\varepsilon,s) \le 2^m \le 8 \Gamma_{\beta,\lambda} \varepsilon^{-\lambda},$$ 
and this bound holds uniformly in $s$. In the field of information-based complexity, such a behaviour is called {\it strong polynomial tractability}, see \cite{NW08,NW10,NW12}. In addition, criteria for other notions of tractability, such as, e.g., polynomial- or weak tractability, can be easily deduced. For example for product weights $\gamma_{\uu}=\prod_{j \in \uu}\gamma_j$ for positive reals $\gamma_j$, $j \in \NN$, condition \eqref{cond:w:spt} is equivalent to $$\sum_{j=1}^{\infty} \gamma_j^{\lambda} < \infty.$$

\begin{remark}
The component-by-component construction from Section~3.4 of \cite{DKLNS14} can be used in the same manner with $b = 2$ and $\alpha = 1+\beta$. In order to be able to efficiently calculate the error criterion $B_s$ we note that 
\begin{align*}
\sum_{\bsk_\vv \in P^\perp_{\vv}(\bsq)} \frac{1}{2^{ (1+\beta) \mu(\bsk_{\vv}, \bszero)}} = & \frac{1}{2^m} \sum_{n=0}^{2^m-1} \sum_{\bsk \in \NN^{|\vv|}} \frac{1}{2^{(1+\beta) \mu(\bsk)}}  \wal_{\bsk}  (\bsy_{n,\vv})=\frac{1}{2^m} \sum_{n=0}^{2^m-1} \prod_{j \in \vv} \Theta(y_n^{(j)}),
\end{align*}
where 
\begin{align*}
\Theta(x) := & \sum_{k=1}^{\infty} \frac{1}{2^{(1+\beta) \mu(k)}}  \wal_{k}  (x) 
= \frac{1-2^{\beta \lfloor \log_2(x)\rfloor} (2^{1+\beta}-1)}{2^{1+\beta}(2^{1+\beta}-2)}.
\end{align*}
Compare with Section~3.4 of \cite{DKLNS14}. Note that the components $y_n^{(j)}$ are exclusively of the form $\frac{k}{2^{2m}}+\frac{1}{2^{2m +1}}$ for $k \in \{0,1,\ldots,2^{2m}-1\}$. Thus, the values of $\Theta(\frac{k}{2^{2m}}+\frac{1}{2^{2m +1}})$ for $k \in \{0,1,\ldots,2^{2m}-1\}$ can be pre-computed and stored in a lookup table.
\end{remark}

\section{Weight functions satisfying Assumption~\ref{assumptionA}}\label{sec:ex:ass:A}

In this section we provide two examples of weight functions that satisfy Assumption~\ref{assumptionA}, which we will be useful for PDEs with a finite number of lognormal random coefficients, as shown in Section~\ref{sec:Appl}. For $k \in N_2 \setminus N_1$ we want to estimate
\begin{align}\label{AssA:int}
\sup_{x \in \mathbb{R}} \frac{|\widetilde{W}^{(\varphi)}_{k, m, 2}(x)|}{\varphi(x) \psi(x)} \qquad \mbox{and} \qquad \sup_{x \in \mathbb{R}} \frac{|\widetilde{W}_{k,m,2}^{(\varphi)}(x)\varphi'(x)|}{\varphi^2(x) \omega(x)}. 
\end{align}
Both suprema share the same general form 
$$\Lambda(k,\varphi,\chi) :=\sup_{x \in \mathbb{R}} \frac{|\widetilde{W}^{(\varphi)}_{k, m, 2}(x)|}{\chi(x)}$$ with a weight function $\chi:\RR \rightarrow (0,\infty)$. As a starting point, we introduce a convenient formula in the next lemma.

\begin{lemma}\label{lem:gen:Lambda}
Assume that $\chi: \RR \rightarrow (0,\infty)$ is symmetric around 0, monotonically increasing over $(-\infty,0]$ and monotonically decreasing over $[0,\infty)$. For $k \in N_2 \setminus N_1$ of the form $k = 2^{a_1} + 2^{a_2}$ with $a_1 > a_2 \ge 0$ we have
$$\Lambda(k,\varphi, \chi) \le \frac{1}{2^{a_1+a_2/2+1}} \sup_{0 \le x \le 1/2} \frac{x}{\chi(\Phi^{-1}(\frac{x}{2^{a_2}}))}$$
where the CDF $\Phi$ is defined in \eqref{def:Phi} and $\Phi^{-1}$ is its inverse.
\end{lemma}

\begin{proof}
We begin with a brief preliminary consideration. For $a_1 > a_2 \ge 0$ and $x \in \RR$ we have 
\begin{equation}\label{eq_w_factor}
w_{2^{a_1}+2^{a_2},0,2}(x)=2^{a_2/2} w_{2^{a_1-a_2}+1,0,2}(2^{a_2}x)
\end{equation}
and $$W_{2^{a_1}+2^{a_2},0,2}(x)=\frac{1}{2^{a_2/2}} W_{2^{a_1-a_2}+1,0,2}(2^{a_2}x)$$
and $$\widetilde{W}_{2^{a_1}+2^{a_2},0,2}(x)=\frac{1}{2^{3a_2/2}} \widetilde{W}_{2^{a_1-a_2}+1,0,2}(2^{a_2}x).$$ 
Here, the first formula follows easily from the definition. For the second formula, we use the first one and obtain
\begin{align*}
W_{2^{a_1}+2^{a_2},0,2}(x) = & \int_0^x w_{2^{a_1}+2^{a_2},0,2}(y) \rd y =  2^{a_2/2} \int_0^x w_{2^{a_1-a_2}+1,0,2}(2^{a_2}y)  \rd y  \\
= & \frac{1}{2^{a_2/2}} \int_0^{2^{a_2} x} w_{2^{a_1-a_2}+1,0,2}(z)  \rd z
=  \frac{1}{2^{a_2/2}} W_{2^{a_1-a_2}+1,0,2}(2^{a_2}x).
\end{align*}
After another integration we also obtain the third formula.

Now we start with the actual proof of Lemma~\ref{lem:gen:Lambda}. For any $m \in L_{k,2}$  and $x \in I(0,a_2)$ we have
\begin{equation*}
\widetilde{W}_{k,0,2}(x) = \widetilde{W}_{k,m,2}\left( x + \frac{m}{2^{a_2}} \right).
\end{equation*}
Therefore we have
\begin{align*}
\Lambda(k,\varphi, \chi) = &  \max_{m \in L_{k,2}} \sup_{\Phi(x) \in I(m,a_2)} \frac{|\widetilde{W}_{k,m,2}(\Phi(x))|}{\chi(x)} \\ 
= & \max_{m \in L_{k,2}} \sup_{x \in I(m,a_2)} \frac{|\widetilde{W}_{k,m,2}(x)|}{\chi(\Phi^{-1}(x))} \\ 
= & \max_{m \in L_{k,2}} \sup_{x \in I(0,a_2)}  \frac{|\widetilde{W}_{k,m,2}( x + \frac{m}{2^{a_2}})|}{\chi(\Phi^{-1}(x))} \\
= & \max_{m \in L_{k,2}} \sup_{x \in I(0,a_2)}  \frac{|\widetilde{W}_{k,0,2}(x)|}{\chi(\Phi^{-1}(x))} \\
= & \sup_{x \in I(0,a_2)} \frac{|\widetilde{W}_{k,0,2}(x)|}{\chi(\Phi^{-1}(x))},
\end{align*}
where the last part follows from the assumption that $\chi$ is symmetric around $0$ and monotonically increasing over $(-\infty,0]$. Now the identities formulated at the beginning imply 
\begin{align*}
\Lambda(k,\varphi, \chi) = & \frac{1}{2^{3 a_2/2}} \sup_{x \in I(0,a_2)} \frac{|\widetilde{W}_{2^{a_1-a_2} + 1,0,2}(x2^{a_2})|}{\chi(\Phi^{-1}(x))} \\ = &  \frac{1}{2^{3 a_2/2}} \sup_{0 \le x < 1} \frac{|\widetilde{W}_{2^{a_1-a_2} + 1,0,2}(x)|}{\chi(\Phi^{-1}(\frac{x}{2^{a_2}}))} \\ = & \frac{1}{2^{3 a_2/2}} \sup_{0 \le x \le 1/2} \frac{|\widetilde{W}_{2^{a_1-a_2} + 1,0,2}(x)|}{\chi(\Phi^{-1}(\frac{x}{2^{a_2}}))},
\end{align*}
where the last inequality follows from the fact that $\chi \circ \Phi^{-1}$ and $\widetilde{W}_{2^{a_1-a_2}+1,0,2}$ are symmetric around $1/2$.

Finally, we also have 
$$|\widetilde{W}_{2^{a_1-a_2}+1, 0, 2}(x)| = \left|\int_0^x   W_{2^{a_1-a_2}+1, 0, 2}(y) \rd y\right| \le x \sup_{y \in [0,1]}   |W_{2^{a_1-a_2}+1, 0, 2}(y)|$$
and further, according to Lemma~\ref{leWal},  
\begin{align*}
\sup_{y \in [0,1]} |W_{2^{a_1-a_2}+1, 0, 2}(y)| =& \sup_{y \in [0,1]} \left|\int_0^y w_{2^{a_1-a_2}+1, 0, 2}(t) \rd t\right|\\
=& \sup_{y \in [0,1]} \left|\int_0^y \wal_{2^{a_1-a_2}+1}(t) \rd t\right| = \frac{1}{2^{a_1-a_2+1}}.
\end{align*}
This yields
$$\Lambda(k,\varphi, \chi) \le \frac{1}{2^{a_1+a_2/2+1}} \sup_{0 \le x \le 1/2} \frac{x}{\chi(\Phi^{-1}(\frac{x}{2^{a_2}}))},$$ as claimed.
\end{proof}

\begin{remark} \label{rem_w_local}
In \eqref{eq_w_factor} we use the fact that the functions $w_{k,0,2}$ are localized. This equality does not hold for classical Walsh functions.
\end{remark}

Now we discuss two concrete examples.

\subsection{Normal distribution} Now we give our main concrete example of weight functions satisfying Assumption~\ref{assumptionA}. For $\sigma>0$ let 
\begin{equation}\label{def:phisigma}
\varphi_\sigma(x) :=\frac{1}{\sqrt{2 \pi} \sigma} \, {\rm e}^{-x^2/(2 \sigma^2)}\qquad \mbox{for $x \in \RR$,}
\end{equation}
be the Gaussian PDF with mean 0 and standard deviation $\sigma$. We write $\Phi_\sigma$ for the corresponding CDF and $\Phi_\sigma^{-1}$ for its inverse.

Assume that
\begin{equation*}
\varphi(x) = \varphi_\sigma(x)= \frac{1}{\sqrt{2\pi} \sigma} \, \mathrm{e}^{-x^2 / (2 \sigma^2) },
\end{equation*}
then 
\begin{equation*}
\varphi'(x)=\varphi'_\sigma(x) = -\frac{x}{\sqrt{2\pi} \sigma^3 } \, \mathrm{e}^{-x^2/ (2 \sigma^2) } =- \frac{x}{\sigma^2} \varphi_\sigma(x).
\end{equation*}
In this case we have $$\frac{|\varphi_\sigma'(x)|}{\varphi_\sigma^2(x)}=\frac{|x|}{\sigma^2} \frac{1}{\varphi_\sigma(x)}.$$ Note that we have $\Phi^{-1}_\sigma(z) \le 0$ if $z \in (0, 1/2]$ and $\Phi^{-1}_\sigma(z) \ge 0$ if $z \in [1/2,1)$.

In the following we choose 
\begin{equation}\label{eq_psi_1}
\psi=\omega=\frac{\varphi_\rho}{\varphi_\sigma}\qquad \mbox{with $\rho\ge \sigma$.} 
\end{equation}
Considering \eqref{AssA:int}, we are, therefore, interested in the quantities
\begin{align*}
\Lambda(k,\varphi_{\sigma}, \varphi_{\rho})=\sup_{x \in \mathbb{R}} \frac{|\widetilde{W}^{(\varphi_\sigma)}_{k, m, 2}(x)|}{\varphi_\rho(x)} \qquad \mbox{and} \qquad \Lambda\left(k,\varphi_{\sigma}, \frac{\varphi_{\sigma} \varphi_{\rho}}{|\varphi_{\sigma}'|}\right)=\frac{1}{\sigma^2}\sup_{x \in \mathbb{R}}  \frac{|x\,\widetilde{W}_{k,m,2}^{(\varphi_{\sigma})}(x)|}{\varphi_\rho(x)}. 
\end{align*}
According to Lemma~\ref{lem:gen:Lambda} we have
\begin{align*}
\Lambda(k,\varphi_{\sigma}, \varphi_{\rho}) \le & \frac{1}{2^{a_1+a_2/2+1}} \sup_{0 \le x \le 1/2} \frac{x}{\varphi_\rho(\Phi^{-1}_\sigma(\frac{x}{2^{a_2}}))} \\ = &  \frac{\sqrt{2 \pi} \, \rho}{2^{a_1+a_2/2+1}} \sup_{0 \le x \le 1/2} x \, {\rm e}^{(\Phi^{-1}_\sigma(\frac{x}{2^{a_2}}))^2/(2 \rho^2)}.
\end{align*}
Next, we use the estimate\footnote{The proof of this estimate is an easy calculation and can be found, for example, in \citet[pp.~332]{DKP22}.} 
\begin{equation*}
0 \le -\Phi_{\sigma}^{-1}(t) \le \sqrt{-2 \sigma^2 \log t} \qquad \mbox{for all $t \in (0,1/2)$,}
\end{equation*}
from which we obtain $(\Phi_{\sigma}^{-1}(t))^2 \le 2 \sigma^2 \log \frac{1}{t}$ for $t \in (0,1/2)$. Hence, 
$${\rm e}^{(\Phi^{-1}_\sigma(\frac{x}{2^{a_2}}))^2/(2\rho^2)} \le {\rm e}^{\tfrac{2 \sigma^2}{2\rho^2} \log \left(\frac{2^{a_2}}{x}\right)}=\frac{2^{a_2 \sigma^2 / \rho^2}}{x^{\sigma^2 / \rho^2}}.$$ This yields
\begin{eqnarray}\label{eq_beta1}
\Lambda(k,\varphi_{\sigma}, \varphi_{\rho}) & \le &  \frac{\sqrt{2 \pi} \, \rho}{2^{a_1+a_2/2+1}} \sup_{0 \le x \le 1/2} x \, \frac{2^{a_2 \sigma^2 / \rho^2}}{x^{\sigma^2/\rho^2}}\nonumber\\
& = & \frac{\sqrt{\pi/2} \, \rho}{2^{a_1+a_2(1/2- \sigma^2 /\rho^2)}} \sup_{0 \le x \le 1/2} x^{1- \sigma^2 / \rho^2}\nonumber\\
& = & \frac{2^{\sigma^2 / \rho^2} \sqrt{\pi/8} \, \rho}{2^{a_1+a_2(1/2-\sigma^2/\rho^2)}}\nonumber\\
& = & 2^{\sigma^2 / \rho^2} \sqrt{\pi/8} \, \rho\, \frac{2^{a_2/2}}{2^{a_1+\beta_1 a_2}} \qquad \mbox{with $\beta_1:=1- \frac{\sigma^2}{\rho^2}$.}
\end{eqnarray}
Again, according to Lemma~\ref{lem:gen:Lambda} we have
\begin{equation*}
\Lambda\left(k,\varphi_\sigma, \frac{\varphi_{\sigma} \varphi_{\rho}}{|\varphi_{\sigma}'|}\right) = \frac{1}{2^{a_1+a_2/2+1}} \frac{1}{\sigma^2}\sup_{0 \le x \le 1/2} \frac{x}{\varphi_\rho(\Phi_{\sigma}^{-1}(\frac{x}{2^{a_2}}))} \left|\Phi_{\sigma}^{-1}\left(\frac{x}{2^{a_2}}\right)\right|.
\end{equation*}
Now we proceed in a similar way to the previous case in order to obtain 
\begin{eqnarray*}
\Lambda\left(k,\varphi_\sigma, \frac{\varphi_{\sigma} \varphi_{\rho}}{|\varphi_{\sigma}'|}\right) & \le & \frac{\sqrt{2 \pi}\, \rho}{2^{a_1+a_2/2+1}} \frac{1}{\sigma^2} \sup_{0 \le x \le 1/2} x \frac{2^{a_2 \sigma^2/\rho^2}}{x^{\sigma^2/\rho^2}} \sqrt{2 \log \left(\frac{2^{a_2}}{x}\right)}\\
& = & \frac{\sqrt{\pi} \, \rho}{2^{a_1+a_2(1/2-\sigma^2/\rho^2)}} \frac{1}{\sigma^2} \sup_{0 \le x \le 1/2} x^{1-\sigma^2/\rho^2} \sqrt{\log \left(\frac{2^{a_2}}{x}\right)}.
\end{eqnarray*}
For every $\delta>0$ we have $\sqrt{\log y} \le (2 \delta \mathrm{e})^{-1/2} y^{\delta}$ for all $y \ge 1$. Hence
\begin{eqnarray}\label{eq_beta2}
\Lambda\left(k,\varphi_{\sigma},\frac{\varphi_{\sigma} \varphi_{\rho}}{|\varphi_{\sigma}'|}\right) & \le & \frac{ \sqrt{\pi} \, \rho}{\sqrt{2 \delta \mathrm{e} }  \, \sigma^2 2^{a_1+a_2(1/2-\sigma^2/\rho^2-\delta)}} \sup_{0 \le x \le 1/2} x^{1-\sigma^2/\rho^2-\delta}\nonumber\\ 
& = & \frac{ \sqrt{\pi} \rho \, 2^{\sigma^2/\rho^2-1+\delta}}{ \sqrt{2\delta \mathrm{e}} \, \sigma^2 2^{a_1+a_2(1/2-\sigma^2/\rho^2-\delta)}}\nonumber\\
& = & \frac{ \sqrt{\pi} \rho \, 2^{\sigma^2/\rho^2-1+\delta}}{ \sqrt{2 \delta \mathrm{e}} \, \sigma^2} \frac{2^{a_2/2}}{2^{a_1+\beta_2 a_2}} \qquad \mbox{for $\beta_2:=1-\frac{\sigma^2}{\rho^2}-\delta$,}
\end{eqnarray}
where the penultimate equality holds true for every $\delta \in (0,1-\frac{\sigma^2}{\rho^2})$. 
Combining \eqref{eq_beta1} and \eqref{eq_beta2} we obtain: 

\begin{proposition}\label{pr:beta}
Let $\varphi$ be the probability density function in \eqref{def:int} and $\psi, \omega$ be the weight functions used in the definition of the norm \eqref{eq_uvnorm}, \eqref{Fnorm:infty}. Let $\varphi=\varphi_\sigma$ and $\psi=\omega=\frac{\varphi_\rho}{\varphi_\sigma}$ with $\rho \ge \sigma$. Then the corresponding $\beta$ in Assumption~\ref{assumptionA} may be chosen to be 
\begin{equation*}
\beta=1-\frac{\sigma^2}{\rho^2}-\delta
\end{equation*}
where $\delta\in (0,1-\frac{\sigma^2}{\rho^2})$ can be chosen arbitrarily close to 0, and where the constant $C > 0$ in Assumption~\ref{assumptionA} can be chosen to be
\begin{equation*}
C = \max\left( 2^{\sigma^2/\rho^2} \sqrt{\pi/8} \rho, \frac{ \sqrt{\pi} \rho 2^{\sigma^2/\rho^2 - 1 + \delta} }{ \sqrt{2\delta \mathrm{e} } \, \sigma^2}  \right).
\end{equation*}
\end{proposition}

\subsection{Student's $t$-distribution} 

The PDF of Student's $t$-distribution with $\nu$ degrees of freedom ($\nu>0$) is given by $$f_{\stu,\nu}(x):=\frac{c_\nu}{\left(1+\frac{x^2}{\nu}\right)^{(\nu+1)/2}}\qquad \mbox{for $x\in \RR$,}$$ where the normalizing factor is $c_\nu:=\frac{\Gamma\left(\frac{\nu+1}{2}\right)}{\sqrt{\pi \nu} \, \Gamma\left(\frac{\nu}{2}\right)}$ for $\nu>0$. Let $\Phi_{\stu,\nu}$ denote the corresponding CDF. The special case $\nu=1$ covers the Cauchy distribution case. In the limit case $\nu=\infty$ we have $f_{\stu,\infty}(x)=\frac{1}{\sqrt{2 \pi}} {\rm e}^{-x^2/2}$, the PDF of the standard normal distribution.

We consider the case where $\varphi(x)=f_{\stu,\nu}(x)$. Choosing $\psi(x)=f_{\stu,\nu}^{\eta-1}(x)$ where $\eta \in [0,\tfrac{\nu}{\nu+1}]$ we obtain $$\sup_{x \in \mathbb{R}} \frac{|\widetilde{W}^{(\varphi)}_{k, m, 2}(x)|}{\varphi(x) \psi(x)}=\sup_{x \in \mathbb{R}} \frac{|\widetilde{W}^{(f_{\stu,\nu})}_{k, m, 2}(x)|}{f_{\stu,\nu}^{\eta}(x)}=\Lambda(k,f_{\stu,\nu},f_{\stu,\nu}^{\eta}).$$ Using Lemma~\ref{lem:gen:Lambda} with $\chi = f_{\stu,\nu}^\eta$ we obtain  
\begin{equation}\label{eq:stud1}
\sup_{x \in \mathbb{R}} \frac{|\widetilde{W}^{(\varphi)}_{k, m, 2}(x)|}{\varphi(x) \psi(x)} \le \frac{1}{2^{a_1+a_2/2 + 1}} \sup_{0 \le x \le 1/2} \frac{x}{\left(f_{\stu,\nu}\left(\Phi_{\stu,\nu}^{-1}\left( \frac{x}{2^{a_2}} \right)\right)\right)^{\eta}}.
\end{equation}

We require some estimates and relations. For $\nu>0$ let $$f_{\rat,\nu}(x):=\frac{\nu}{2} \frac{1}{(1+|x|)^{\nu+1}}\qquad \mbox{for $x\in \RR$,}$$ be a rational PDF and let $\Phi_{\rat,\nu}$ denote the corresponding CDF. From \cite[Eq.~(15)]{KSWW10} we know that $$\Phi_{\rat,\nu}^{-1}\left(\frac{u}{A_\nu}\right) \le \Phi_{\stu,\nu}^{-1}(u) \qquad \mbox{for $u \in (0,1)$},$$ where $A_\nu:= \frac{2 c_\nu}{\nu} \, (\nu+1)^{(\nu+1)/2} > 1$. Furthermore, from \cite[Table~1]{KSWW10} we know that for $u \in (0,\tfrac{1}{2}]$ we have $$\Phi_{\rat,\nu}^{-1}(u)=1-(2 u)^{-1/\nu}.$$

Now we continue estimating \eqref{eq:stud1}. For $x \in [0,\frac{1}{2}]$ we have $\frac{x}{2^{a_2}}\le \frac{1}{2}$ and therefore $\Phi^{-1}_{\stu,\nu}(\frac{x}{2^{a_2}}) \le 0$. Thus $$\Phi_{\rat,\nu}^{-1}\left(\frac{x}{A_\nu \, 2^{a_2}}\right) \le \Phi_{\stu,\nu}^{-1}\left(\frac{x}{2^{a_2}}\right)\le 0.$$ Consequently, $$\left(\Phi_{\stu,\nu}^{-1}\left(\frac{x}{2^{a_2}}\right)\right)^2 \le \left(\Phi_{\rat,\nu}^{-1}\left(\frac{x}{A_\nu \, 2^{a_2}}\right) \right)^2=\left(1-\left(\frac{A_\nu \, 2^{a_2}}{2 x}\right)^{1/\nu}\right)^2.$$ Therefore, for $x \in (0,\frac{1}{2}]$ we have 
\begin{eqnarray*}
\frac{x}{\left(f_{\stu,\nu}\left(\Phi_{\stu,\nu}^{-1}\left( \frac{x}{2^{a_2}} \right)\right)\right)^{\eta}} & = & \frac{x}{c_{\nu}^{\eta}} \left(1+\frac{1}{\nu} \left(\Phi_{\stu,\nu}^{-1}\left(\frac{x}{2^{a_2}}\right)\right)^2\right)^{\eta (\nu+1)/2}   \\
& \le & \frac{x}{c_{\nu}^{\eta}} \left(1+\frac{1}{\nu} \left(1-\left(\frac{A_\nu \, 2^{a_2}}{2 x}\right)^{1/\nu}\right)^2\right)^{\eta (\nu+1)/2}.
\end{eqnarray*}
For the expression in the brackets we have
\begin{eqnarray*}
1+\frac{1}{\nu} \left(1-\left(\frac{A_\nu \, 2^{a_2}}{2 x}\right)^{1/\nu}\right)^2 & = & 1+\frac{1}{\nu}-\frac{2}{\nu} \left(\frac{A_\nu \, 2^{a_2}}{2 x} \right)^{1/\nu} + \frac{1}{\nu} \left(\frac{A_\nu \, 2^{a_2}}{2 x} \right)^{2/\nu}\\
& \le & 2^{2 a_2/\nu} \left(\frac{\nu+1}{\nu}+\frac{2}{\nu} \left(\frac{A_\nu}{2 x} \right)^{1/\nu} + \frac{1}{\nu} \left(\frac{A_\nu}{2 x} \right)^{2/\nu}\right).
\end{eqnarray*}
Consequently,
\begin{eqnarray*}
\frac{x}{\left(f_{\stu,\nu}\left(\Phi_{\stu,\nu}^{-1}\left( \frac{x}{2^{a_2}} \right)\right)\right)^{\eta}} & \le & \frac{x}{c_{\nu}^{\eta}} \left(2^{2 a_2/\nu} \left(\frac{\nu+1}{\nu}+\frac{2}{\nu} \left(\frac{A_\nu}{2 x} \right)^{1/\nu} + \frac{1}{\nu} \left(\frac{A_\nu}{2 x} \right)^{2/\nu}\right)\right)^{\eta (\nu+1)/2} \\
& = & 2^{a_2 \,  \frac{\nu+1}{\nu}\, \eta} \frac{x}{c_{\nu}^{\eta}} \left(\frac{\nu+1}{\nu}+\frac{2}{\nu} \left(\frac{A_\nu}{2 x} \right)^{1/\nu} + \frac{1}{\nu} \left(\frac{A_\nu}{2 x} \right)^{2/\nu}\right)^{\eta(\nu+1)/2}.
\end{eqnarray*}

Returning to \eqref{eq:stud1}, for $\eta \le \frac{\nu}{\nu+1}$ we obtain
$$\sup_{x \in \mathbb{R}} \frac{|\widetilde{W}^{(\varphi)}_{k, m, 2}(x)|}{\varphi(x) \psi(x)} \le C_{\nu,\eta} \, \frac{1}{2^{a_1+a_2/2}} \, 2^{a_2 \,  \frac{\nu+1}{\nu}\, \eta} = C_{\nu,\eta} \, \frac{2^{a_2/2}}{2^{a_1+a_2(1-\frac{\nu+1}{\nu}\, \eta)}},$$ where 
\begin{equation}\label{def:cnueta}
C_{\nu,\eta}:=\frac{1}{2 c_{\nu}^{\eta}} \sup_{0\le x \le 1/2} x \left(\frac{\nu+1}{\nu}+\frac{2}{\nu} \left(\frac{A_\nu}{2 x} \right)^{1/\nu} + \frac{1}{\nu} \left(\frac{A_\nu}{2 x} \right)^{2/\nu}\right)^{\eta(\nu+1)/2}< \infty,
\end{equation}
for $\eta \le \frac{\nu}{\nu+1}$, because 
$$x \left(\frac{\nu+1}{\nu}+\frac{2}{\nu} \left(\frac{A_\nu}{2 x} \right)^{1/\nu} + \frac{1}{\nu} \left(\frac{A_\nu}{2 x} \right)^{2/\nu}\right)^{\eta(\nu+1)/2}=O\left(x^{1-\frac{\nu+1}{\nu}\eta} \right).$$

Now we treat the second supremum from \eqref{AssA:int}. We have $f_{\stu,\nu}'(x)=-\frac{\nu+1}{\nu} \frac{x}{1+\frac{x^2}{\nu}} f_{\stu,\nu}(x)=-\frac{\nu+1}{\nu} \frac{x}{c_{\nu}^{2/(\nu+1)}} (f_{\stu,\nu}(x))^{1+\frac{2}{\nu+1}}$ and hence, choosing $\varphi=f_{\stu,\nu}$, we have $$\frac{|\varphi'(x)|}{\varphi^2(x)}= \frac{\nu+1}{\nu} \frac{1}{c_{\nu}^{2/(\nu+1)}}\frac{|x|}{(f_{\stu,\nu}(x))^{1-\frac{2}{\nu+1}}}.$$ Choosing $\omega(x)=\frac{\nu+1}{\nu} \frac{1}{c_{\nu}^{2/(\nu+1)}}\frac{|x|}{(f_{\stu,\nu}(x))^{1-\frac{2}{\nu+1}-\eta}}$ we obtain
$$\sup_{x \in \mathbb{R}} \frac{|\widetilde{W}_{k,m,2}^{(\varphi)}(x)\varphi'(x)|}{\varphi^2(x) \omega(x)} = \sup_{x \in \mathbb{R}} \frac{|\widetilde{W}_{k,m,2}^{(f_{\stu,\nu})}(x)|}{f_{\stu,\nu}^{\eta}(x)}=\Lambda(k,f_{\stu,\nu},f_{\stu,\nu}^{\eta}),$$ and we can proceed as above. Thus, we obtain the following proposition:

\begin{proposition}\label{pr:beta3}
Let $\varphi$ be the probability density function in \eqref{def:int} and $\psi, \omega$ be the weight functions used in the definition of the norm \eqref{eq_uvnorm}, \eqref{Fnorm:infty}. Let $\varphi(x)=f_{\stu,\nu}(x)$ and $\psi(x)=f_{\stu,\nu}^{\eta-1}(x)$ and $\omega(x)=\frac{\nu+1}{\nu} \frac{1}{c_{\nu}^{2/(\nu+1)}}|x|f^{\eta-1+\frac{2}{\nu+1}}_{\stu,\nu}(x)$ with $\eta \in [0,\frac{\nu}{\nu+1}]$. Then the corresponding $\beta$ in Assumption~\ref{assumptionA} may be chosen to be $$\beta=1-\frac{\nu+1}{\nu}\, \eta,$$ with a corresponding $C=C_{\nu,\eta}$ given by \eqref{def:cnueta}.
\end{proposition}

\section{Application: Elliptic PDEs with log-normal random coefficients}\label{sec:Appl}

In the following, we apply the error bounds to a partial differential equation with a finite number of random log-normal coefficients, see \cite{GKNSSS15} for the analogue case with an infinite number of random coefficients. The formulation of this problem requires many concepts and assumptions which would exceed the scope of this work. We limit ourselves to a concise presentation of the most important points so that our concrete application can be followed. A gentle overview can also be found in Appendix~A of \cite{DKP22}.

\subsection{The PDE problem setting}

Consider the diffusion problem
\begin{equation}\label{def:difprobl}
-\nabla_{\bsx} \cdot (a(\bsx, \omega) \nabla_{\bsx} \, u(\bsx, \omega)) = f(\bsx), \quad \mbox{for almost all $ \omega \in \Omega$ and $\bsx \in D$,}
\end{equation}
subject to the homogeneous Dirichlet boundary condition $u(\bsx,\omega) = 0$ for $\bsx \in \partial D$, where $D$ is a bounded domain in $\RR^d$, $(\Omega, \mathcal{A}, \mathbb{P})$ is a probability space, and $\nabla_{\bsx}=(\partial/\partial x_1,\ldots,\partial/\partial x_d)^{\top}$ is the gradient operator. The function $a$ is a lognormal random field given by
\begin{equation*}
a(\bsx, \omega) = a_\ast(\bsx) + a_0(\bsx) \exp(Z(\bsx, \omega)),
\end{equation*}
where $a_\ast$, $a_0$ are given functions that are continuous on $\overline{D}$ with $a_\ast$ non-negative and $a_0$ strictly positive on $\overline{D}$ and $Z$ is a zero-mean Gaussian random field given by a finite Karhunen-Lo\'eve expansion
\begin{equation*}
Z(\bsx, \omega) = \sum_{j=1}^s \sqrt{\mu_j} \, \xi_j(\bsx) Y_j(\omega), \quad \bsx \in D,
\end{equation*} 
for some fixed $s \in \mathbb{N}$ and with i.i.d. $N(0,1)$-distributed random variables $(Y_j)_{j \ge 1}$. For more background, assumptions and the case where $s= \infty$, see \cite{GKNSSS15}. The focus in the following is on the integration problem arising from approximating the expected value $\EE[G(u_h^s)]$ of a linear functional $G$ of an approximation $u_h^s$ of the solution $u$ of the PDE problem, where the expectation is taken with respect to random $\omega \in \Omega$ ($u_h^s$ is a finite element approximation with mesh width $h$ of the PDE problem, see Section~2 of \cite{GKNSSS15} for details). Using the same assumptions and setting as in \cite{GKNSSS15}, we show that if one uses importance sampling, then we can apply Theorem~\ref{thm:error_bound}. Since we consider a fixed, finite number of random coefficients, we do not include the truncation error, and we do not include the dependence on the dimension of the truncation error. Theorem~\ref{thm:error_bound} shows a convergence order bigger than $1$, but does not allow us to obtain a bound on the integration error of order bigger than $1$ which is also independent of the dimension.

In order to be able to apply Theorem~\ref{thm:error_bound}, we first need bounds on the derivatives up to order $2$ in each variable of a linear functional of the solution $u$ with respect to the random coefficients. In the following, we show an analogue of Theorem~16 in \cite{GKNSSS15} with the modified setting as above and with the norm given by \eqref{Fnorm:infty}.

The bound in Theorem~14 in \cite{GKNSSS15} states that
\begin{equation*}
\left\| \frac{\partial^{|\boldsymbol{\nu}|} u_h^s(\cdot, \bsy)}{\partial \bsy_{\bsnu}} \right\|_V \le \frac{|\boldsymbol{\nu}|!}{(\ln 2)^{|\boldsymbol{\nu}|}} \left( \prod_{j = 1}^s b_j^{\nu_j} \right) \frac{\|f\|_{V'}}{\check{a}(\bsy)},
\end{equation*}
for any vector $\boldsymbol{\nu} = (\nu_1, \nu_2, \ldots, \nu_s ) \in \NN_0^{s}$ with $|\boldsymbol{\nu}| = \sum_{j=1}^s \nu_j$, where $V$ is an appropriate normed function space on $D$, equipped with $\|\cdot\|_V$, and $V'$ is the corresponding dual space, the function $f$ from the right-hand side in \eqref{def:difprobl} belonging to $V'$, and $\check{a}(\bsy) := \min_{\bsx \in \overline{D}} a(\bsx, \bsy)$. 

The integrand of interest is given by
\begin{equation}\label{F_PDE}
F(\bsy) := G(u_h^s(\cdot, \bsy)). 
\end{equation}
By the linearity of $G$ we obtain
\begin{equation*}
\left| \frac{\partial^{|\boldsymbol{\nu}|} F}{\partial \bsy_{\boldsymbol{\nu}}}(\bsy)  \right| \le \|G\|_{V'} \left\| \partial^{\boldsymbol{\nu}} u_h^s(\cdot, \bsy) \right\|_V.
\end{equation*}
Since $\min_{ \bsx \in \overline{D}} \sqrt{\mu_j} \xi_j(\bsx) \ge - b_j$, there is a constant $C > 0$ such that $\frac{1}{\check{a}(\bsy)} \le C \prod_{j=1}^s \exp( b_j |y_j|)$. Thus, we have
\begin{equation}\label{est:norm:assumpt}
\left|\frac{\partial^{|\boldsymbol{\nu}|} F}{\partial \bsy_{\boldsymbol{\nu}}}(\boldsymbol{y}) \right| \le K^\ast \frac{|\boldsymbol{\nu}|!}{(\ln 2)^{|\boldsymbol{\nu}|}} \left( \prod_{j =1}^s b_j^{\nu_j} \right) \left( \prod_{j=1}^s \exp(b_j |y_j|) \right),
\end{equation}
for some constant $K^\ast$, which also absorbs the norms $\|f\|_{V'}$ and $\|G\|_{V'}$.

\subsection{Bound on the integration error for PDEs with log-normal coefficients via importance sampling}

For the integration of the function $F$ \eqref{F_PDE} with respect to $\varphi_1$ (the Gaussian PDF with mean 0 and standard deviation the identity matrix, cf.~\eqref{def:phisigma}) we employ importance sampling using $\varphi_\sigma$ with $\sigma > 1$. Recall that in the multivariate case we consider $\varphi_{\sigma}(\bsx)=\prod_{j=1}^s \varphi_{\sigma}(x_j)$. Then the integrand becomes 
\begin{equation}\label{F_sigma_PDE}
F_\sigma:=F \, \frac{\varphi_1}{\varphi_\sigma}
\end{equation}
and we integrate this function with respect to $\varphi_\sigma$ in each coordinate.

\begin{theorem}\label{thm:pde_app}
Consider the approximation of the integral $$\int_{\mathbb{R}^s} F(\bsy) \varphi_1(\bsy) \,\mathrm{d} \bsy = \int_{\mathbb{R}^s} F_\sigma(\bsy) \varphi_\sigma(\bsy) \,\mathrm{d} \bsy, $$ where $F$  is given by \eqref{F_PDE} and $F_\sigma$ is given by \eqref{F_sigma_PDE}, by an interlaced polynomial lattice rule. Let $\rho \ge \sigma > 1$ be such that $1 + \tfrac{1}{\rho^2} - \tfrac{2}{\sigma^2} > 0$ and choose $\psi=\omega=\varphi_{\rho}/\varphi_{\sigma}$ in the definition of the norm \eqref{eq_uvnorm} and \eqref{Fnorm:infty}.
Then for every $m$ a shifted interlaced polynomial lattice rule of order $2$ with underlying point set $\cP_{s,m}^{{\rm sh}}$ with $N=2^m$ points can be constructed using a component-by-component algorithm minimizing $B_s(\bsq)$ in \eqref{eq:EBsq} such that for all $s \in \NN$, for all $\delta \in (0, 1-\sigma^2/\rho^2)$ and all $\lambda \in (\frac{1}{1+\beta},1]$, where $\beta = 1 - \sigma^2/\rho^2 - \delta$, there exists a quantity $C(s,\delta,\lambda) > 0$, such that the integration error satisfies 
\begin{equation*}
\left| \int_{\mathbb{R}^s} F_\sigma(\bsy) \varphi_\sigma(\bsy) \,\mathrm{d} \bsy - \frac{1}{N} \sum_{\bsx \in \cP_{s,m}^{{\rm sh}}} F_\sigma(\Phi_\sigma^{-1}(\bsy)) \right| \le \frac{C(s,\delta,\lambda)}{N^{1/\lambda}} 
\|F_\sigma\|_{\bsgamma, \varphi_\sigma, \varphi_\rho/\varphi_\sigma, \varphi_\rho/\varphi_\sigma}.
\end{equation*}

\end{theorem}

\begin{remark}
Note that $\beta$ can be chosen arbitrarily close to $1-\sigma^2/\rho^2$ and $\lambda$ arbitrarily close to $1/(1+\beta)$. Hence, the convergence rate can be arbitrarily close to $N^{-2+\sigma^2/\rho^2}$.  The smaller the quotient $\sigma^2/\rho^2$, the better the rate.
\end{remark}

In the following lemma we show that the norm of $F_\sigma$ is finite. Then Theorem~\ref{thm:pde_app} follows from Theorem~\ref{thm:error_bound}.

\begin{lemma}\label{lem:bound_normF}
Let $F$ be given by \eqref{F_PDE} and $F_\sigma$ by \eqref{F_sigma_PDE}. Let $\rho, \sigma, \bsgamma$ be defined as in Theorem~\ref{thm:pde_app}. Assume that $\rho \ge \sigma > 1$ and that $1 + \tfrac{1}{\rho^2} - \tfrac{2}{\sigma^2} > 0$. Then there exists a constant $C> 0$ depending on  $(b_j)_j$, $\rho, \sigma, s, K^*$, such that
\begin{equation*}
\left\| F_{\sigma} \right\|_{\bsgamma,\varphi_\sigma,\varphi_\rho/\varphi_\sigma,\varphi_\rho / \varphi_\sigma} \le C < \infty. 
\end{equation*}
\end{lemma}

\begin{proof}
It suffices to show that the $\|F_\sigma\|_{\uu,\vv}$ is finite. We now choose the weight functions $\psi=\omega = \varphi_\rho / \varphi_\sigma$ with $\rho \ge \sigma > 1$. First we derive  a bound on the partial derivatives of the integrand $F_\sigma$. We have
$$\frac{\partial^{|(\bsone_{\uu \setminus \ww},\bstwo_{\ww},\bszero)|}F_{\sigma}}{\partial \bsx_{(\bsone_{\uu \setminus \ww},\bstwo_{\ww},\bszero)}} (\bsx) = \sum_{\bsnu \le (\bsone_{\uu \setminus \ww},\bstwo_{\ww},\bszero)} {(\bsone_{\uu \setminus \ww},\bstwo_{\ww},\bszero) \choose \bsnu} \frac{\partial^{|\bsnu|}F}{\partial \bsx_{\bsnu}}(\bsx) \frac{\partial^{|(\bsone_{\uu \setminus \ww},\bstwo_{\ww},\bszero) -\bsnu|}}{\partial \bsx_{(\bsone_{\uu \setminus \ww},\bstwo_{\ww},\bszero)-\bsnu}} \frac{\varphi_1(\bsx)}{\varphi_{\sigma}(\bsx)},$$
where the sum is over all $s$-tuples $\bsnu$ that are coordinate-wise less than or equal to $(\bsone_{\uu \setminus \ww},\bstwo_{\ww},\bszero)$ and where, for $\bsnu=(\nu_1,\ldots,\nu_s)$ and $\bsk=(k_1,\ldots,k_s)$, we use ${\bsk \choose \bsnu}:=\prod_{j=1}^s {k_j \choose \nu_j}$. We have $$\frac{\rd}{\rd x} \frac{\varphi_1(x)}{\varphi_{\sigma}(x)}=-\frac{x}{\sigma/(\sigma^2-1)} \exp\left(-\frac{x^2}{2 \sigma^2/(\sigma^2-1)}\right)$$ 
and 
$$\frac{\rd^2}{\rd x^2} \frac{\varphi_1(x)}{\varphi_{\sigma}(x)}=\frac{1}{\sigma/(\sigma^2-1)} \left(\frac{x^2}{\sigma^2/(\sigma^2-1)}-1\right) \exp\left(-\frac{x^2}{2 \sigma^2/(\sigma^2-1)}\right).$$ Hence
\begin{eqnarray*}
\left|\frac{\partial^{|(\bsone_{\uu \setminus \ww},\bstwo_{\ww},\bszero) -\bsnu|}}{\partial \bsx_{(\bsone_{\uu \setminus \ww},\bstwo_{\ww},\bszero)-\bsnu}} \frac{\varphi_1(\bsx)}{\varphi_{\sigma}(\bsx)}\right| & \le & \prod_{j = 1}^s \exp\left(-\frac{x_j^2}{2 \sigma^2/(\sigma^2-1)}\right) \left| \prod_{j=1}^s p_{((\bsone_{\uu \setminus \ww},\bstwo_{\ww},\bszero) -\bsnu)_j}(x_j)\right|,
\end{eqnarray*}
where $((\bsone_{\uu \setminus \ww},\bstwo_{\ww},\bszero) -\bsnu)_j$ is the $j$-th entry of the vector $(\bsone_{\uu \setminus \ww},\bstwo_{\ww},\bszero) -\bsnu$, which is either $0$, $1$ or $2$, and where
$$p_{\ell}(x) := \left\{ 
\begin{array}{ll}
\sigma & \mbox{if } \ell=0,\\
-\frac{x}{\sigma/(\sigma^2-1)} & \mbox{if } \ell=1,\\
\frac{1}{\sigma/(\sigma^2-1)} \left(\frac{x^2}{\sigma^2/(\sigma^2-1)}-1\right) & \mbox{if } \ell=2.
\end{array}\right. $$
Next we also use \eqref{est:norm:assumpt} and estimate
\begin{eqnarray*}
\lefteqn{\left|\frac{\partial^{|(\bsone_{\uu \setminus \ww},\bstwo_{\ww},\bszero)|}F_{\sigma}}{\partial \bsx_{(\bsone_{\uu \setminus \ww},\bstwo_{\ww},\bszero)}} (\bsx)\right| }\\
& \le &  \sum_{\bsnu \le (\bsone_{\uu \setminus \ww},\bstwo_{\ww},\bszero)} {(\bsone_{\uu \setminus \ww},\bstwo_{\ww},\bszero) \choose \bsnu} \left|\frac{\partial^{|\bsnu|}F}{\partial \bsx_{\bsnu}}(\bsx)\right|  \\
& & \times \prod_{j = 1}^s \exp\left(-\frac{x_j^2}{2 \sigma^2/(\sigma^2-1)}\right) \left| \prod_{j=1}^s p_{((\bsone_{\uu \setminus \ww},\bstwo_{\ww},\bszero) -\bsnu)_j}(x_j)\right|\\
 & \le & K^{\ast} \sum_{\bsnu \le (\bsone_{\uu \setminus \ww},\bstwo_{\ww},\bszero)} {(\bsone_{\uu \setminus \ww},\bstwo_{\ww},\bszero) \choose \bsnu} \frac{|\bsnu|!}{(\ln 2)^{|\bsnu|}} \left(\prod_{j=1}^s b_j^{\nu_j}\right) \left(\prod_{j=1}^s \exp(b_j|x_j|)\right)\\
& & \times \prod_{j = 1}^s \exp\left(-\frac{x_j^2}{2 \sigma^2/(\sigma^2-1)}\right) \left| \prod_{j=1}^s p_{((\bsone_{\uu \setminus \ww},\bstwo_{\ww},\bszero) -\bsnu)_j}(x_j)\right|\\
 & = & K^{\ast} \prod_{j = 1}^s \exp\left(-\frac{x_j^2}{2 \sigma^2/(\sigma^2-1)} +b_j|x_j|\right) \\
 && \times \sum_{\bsnu \le (\bsone_{\uu \setminus \ww},\bstwo_{\ww},\bszero)} {(\bsone_{\uu \setminus \ww},\bstwo_{\ww},\bszero) \choose \bsnu} \frac{|\bsnu|!}{(\ln 2)^{|\bsnu|}} \left(\prod_{j=1}^s b_j^{\nu_j}\right)  \left| \prod_{j=1}^s p_{((\bsone_{\uu \setminus \ww},\bstwo_{\ww},\bszero) -\bsnu)_j}(x_j)\right|.
\end{eqnarray*}
This gives us a suitable estimate for the absolute value of the partial derivatives of $F_\sigma$. Substituting this estimate into $\|\cdot\|_{\uu, \vv}$ defined in \eqref{eq_uvnorm} with $\psi=\omega$, we thus have
\begin{eqnarray*}
\|F_{\sigma} \|_{\uu, \vv} & \le & \sum_{\ww \subseteq \vv} \int_{\mathbb{R}^{|\uu|}} \left| \int_{\mathbb{R}^{s-|\uu|}} \frac{\partial^{|\uu| + |\ww|} F_{\sigma}}{\partial \bsx_\ww \partial \bsx_{\uu}}(\bsx)\prod_{j \in [s]\setminus \uu} \varphi_{\sigma}(x_j) \rd \bsx_{[s] \setminus \uu} \right|   \prod_{j \in \vv} \psi(x_j)     \rd \bsx_{\uu}\\
& \le & K^{\ast} \sum_{\ww \subseteq \vv} \int_{\mathbb{R}^{|\uu|}}  \int_{\mathbb{R}^{s-|\uu|}} \prod_{j = 1}^s \exp\left(-\frac{x_j^2}{2 \sigma^2/(\sigma^2-1)} +b_j|x_j|\right)\\
& & \sum_{\bsnu \le (\bsone_{\uu \setminus \ww},\bstwo_{\ww},\bszero)} {(\bsone_{\uu \setminus \ww},\bstwo_{\ww},\bszero) \choose \bsnu} \frac{|\bsnu|!}{(\ln 2)^{|\bsnu|}} \left(\prod_{j=1}^s b_j^{\nu_j}\right)  \left| \prod_{j=1}^s p_{((\bsone_{\uu \setminus \ww},\bstwo_{\ww},\bszero) -\bsnu)_j}(x_j)\right| \\
& & \prod_{j \in [s]\setminus \uu} \varphi_{\sigma}(x_j) \rd \bsx_{[s] \setminus \uu}   \prod_{j \in \vv} \psi(x_j)     \rd \bsx_{\uu}\\
& = & K^{\ast} \sum_{\ww \subseteq \vv} \sum_{\bsnu \le (\bsone_{\uu \setminus \ww},\bstwo_{\ww},\bszero)} {(\bsone_{\uu \setminus \ww},\bstwo_{\ww},\bszero) \choose \bsnu} \frac{|\bsnu|!}{(\ln 2)^{|\bsnu|}} \left(\prod_{j=1}^s b_j^{\nu_j}\right) \\
&& \times  \prod_{j \in [s]\setminus \uu} \int_{-\infty}^{\infty} \exp\left(-\frac{x^2}{2 \sigma^2/(\sigma^2-1)} +b_j|x|\right) |p_{((\bsone_{\uu \setminus \ww},\bstwo_{\ww},\bszero) -\bsnu)_j}(x)| \varphi_{\sigma}(x) \rd x \\
&& \times \prod_{j \in \vv} \int_{-\infty}^{\infty} \exp\left(-\frac{x^2}{2 \sigma^2/(\sigma^2-1)} +b_j|x|\right) |p_{((\bsone_{\uu \setminus \ww},\bstwo_{\ww},\bszero) -\bsnu)_j}(x)| \psi(x) \rd x \\
&& \times \prod_{j \in \uu \setminus\vv} \int_{-\infty}^{\infty} \exp\left(-\frac{x^2}{2 \sigma^2/(\sigma^2-1)} +b_j|x|\right) |p_{((\bsone_{\uu \setminus \ww},\bstwo_{\ww},\bszero) -\bsnu)_j}(x)|  \rd x.
\end{eqnarray*}
Now we require precise estimates of the integrals that appear in the above estimate. For $j \in [s]\setminus \uu$ we have $p_{((\bsone_{\uu \setminus \ww},\bstwo_{\ww},\bszero) -\bsnu)_j}(x)=p_0(x)=\sigma$. Hence, for the integral with respect to $\varphi_\sigma$ we have 
\begin{align}\label{est:int1}
\lefteqn{\int_{-\infty}^{\infty} \exp\left(-\frac{x^2}{2 \sigma^2/(\sigma^2-1)} +b_j|x|\right) \sigma \varphi_{\sigma}(x) \,\mathrm{d} x}\nonumber\\
&= \frac{1}{\sqrt{2 \pi}} \left[\int_{-\infty}^0 \exp\left(- \frac{x^2}{2} - b_j x\right)  \,\mathrm{d} x +  \int_0^{\infty} \exp\left(-\frac{x^2}{2} +b_j x\right) \,\mathrm{d} x \right]\nonumber\\
&= \frac{1}{\sqrt{2 \pi}} \left[\int_{-\infty}^0 \exp\left(-\frac{(x+b_j)^2}{2} + \frac{b_j^2}{2} \right)  \,\mathrm{d} x +\int_0^{\infty} \exp\left(-\frac{(x-b_j)^2}{2} + \frac{b_j^2}{2} \right)  \,\mathrm{d} x \right]\nonumber\\
&= \frac{1}{\sqrt{2 \pi}} \exp\left(\frac{b_j^2}{2}\right) \left[\int_{-\infty}^0 \exp\left(-\frac{(x+b_j)^2}{2}\right) \,\mathrm{d} x + \int_0^{\infty} \exp\left(-\frac{(x-b_j)^2}{2}\right)   \,\mathrm{d} x\right]\nonumber\\
&= \exp\left(\frac{b_j^2}{2}\right) \left[\int_{-\infty}^{b_j} \varphi_1(x) \,\mathrm{d} x + \int_{-b_j}^{\infty} \varphi_1(x)   \,\mathrm{d} x\right]\nonumber\\
&= \exp\left(\frac{b_j^2}{2}\right) \left[1+\int_{-b_j}^{b_j} \varphi_1(x)   \,\mathrm{d} x\right]\nonumber\\
&\le \exp\left(\frac{b_j^2}{2}\right) (1+2 b_j)\nonumber\\
& \le \exp\left(2 b_j+\frac{b_j^2}{2}\right)\nonumber\\
& \le \exp\left(2 \max(b_j,b_j^2)\right).
\end{align}
Therefore, 
$$\prod_{j \in [s]\setminus \uu} \int_{-\infty}^{\infty} \exp\left(-\frac{x^2}{2 \sigma^2/(\sigma^2-1)} +b_j|x|\right) \sigma \varphi_{\sigma}(x) \,\mathrm{d} x \le \exp\left(2 \sum_{j \in [s]\setminus \uu} \max(b_j,b_j^2)\right).$$
Next, we consider the integral with respect to $\psi$ in the estimate of the norm of $F_{\sigma}$. Recall that we chose $\psi = \varphi_\rho/ \varphi_\sigma$ with $\rho \ge \sigma$, as in \eqref{eq_psi_1}. We have 
\begin{eqnarray*}
\lefteqn{\int_{-\infty}^{\infty} \exp\left(-\frac{x^2}{2 \sigma^2/(\sigma^2-1)} +b_j|x|\right) | p_{((\bsone_{\uu \setminus \ww},\bstwo_{\ww},\bszero) -\bsnu)_j}(x)|  \psi(x) \rd x}  \\ 
& = & \frac{\sigma}{\rho} \int_{-\infty}^\infty \exp\left(-\frac{x^2}{2} \left(1+ \frac{1}{\rho^2} - \frac{2}{\sigma^2} \right)  + b_j |x| \right) |p_{((\bsone_{\uu \setminus \ww},\bstwo_{\ww},\bszero) -\bsnu)_j}(x)|  \rd x.
\end{eqnarray*}
This is finite as long as $1 + \frac{1}{\rho^2} - \frac{2}{\sigma^2} > 0$. This can be shown similarly as \eqref{est:int1}. We also have 
\begin{equation*}
\int_{-\infty}^{\infty}  \exp\left(-\frac{x^2}{2 \sigma^2/(\sigma^2-1)} +b_j|x|\right) |p_{((\bsone_{\uu \setminus \ww},\bstwo_{\ww},\bszero) -\bsnu)_j}(x)|  \rd x < \infty.
\end{equation*}
From here we see that $\|F_{\sigma}\|_{\uu,\vv}$ is indeed finite. \end{proof}

\begin{proof}[Proof of Theorem~\ref{thm:pde_app}]
We apply Theorem~\ref{thm:error_bound}. According to Lemma~\ref{lem:bound_normF} the norm $\left\| F_{\sigma} \right\|_{\infty, \bsgamma,\varphi_\sigma,\varphi_\rho/\varphi_\sigma,\varphi_\rho / \varphi_\sigma}$ is finite. The assumption $1 + \frac{1}{\rho^2} - \frac{2}{\sigma^2} > 0$ is true for $\sigma > \rho \sqrt{\frac{2}{1+\rho^2}}$. Since we also need $\rho \ge \sigma$, we require
\begin{equation}\label{ineq_sig_rho}
\rho \sqrt{\frac{2}{1+\rho^2}} < \sigma \le \rho,
\end{equation}
and $\sigma  > 1$. Now $1 \le \rho \sqrt{\frac{2}{1+\rho^2}} < \rho$ holds for $\rho > 1$. In this case, we get from Proposition~\ref{pr:beta} that $$\beta = 1 - \frac{\sigma^2}{\rho^2}-\delta$$ where $\delta \in (0,1-\sigma^2/\rho^2)$ can be chosen arbitrarily close to 0. Now the result follows from Theorem~\ref{thm:error_bound}.
\end{proof}

\section{Numerical results}\label{sec:numerics}

In this section, we validate our theoretical developments by applying the QMC integration rule \eqref{QMC:rule} to compute expected values of selected quantities of interest (QoIs) arising in two diffusion problems, where the diffusion coefficients are modelled as log-normal random fields. In both examples, we extend slightly beyond the homogeneous Dirichlet boundary condition assumption in Section~\ref{sec:Appl} by considering more challenging mixed boundary conditions. In all numerical experiments, we generate the QMC point sets using the QMC4PDE package of \citeauthor{QMC4PDE} accompanying the article of \cite{KN16}, with the interlaced polynomial lattice rule of order two. 

The first experiment considers the diffusion problem on a one-dimensional domain $D = [0,1]$:
\begin{equation}\label{eq:diff1d}
- \frac{\partial}{\partial x} \Big( a(x, \omega) \frac{\partial}{\partial x} \, u(x, \omega)\Big) = f(x), 
\end{equation}
subject to boundary conditions
\begin{equation}\label{eq:diff1d_bc}
u(0, \omega) =0 \quad \text{and} \quad a(1, \omega) \frac{\partial}{\partial x} u(1, \omega) = R,
\end{equation}
for almost all $\omega \in \Omega$ and $x \in D$. The diffusion field is defined as 
\begin{equation}\label{eq:diff1d_rf}
a(x, \omega) = \exp\Big(a_\ast + \sum_{j = 1}^s \sqrt{\mu_j} \, \xi_j(x) Y_j(\omega)\Big), \quad a_\ast = 1/2, \quad \xi_j(x) = \cos(\pi j x),
\end{equation}
where $\mu_j = j^{-2\alpha}$ with $\alpha =1$ is a sequence of decaying weights, and $\big(Y_j(\omega)\big)_{j = 1, \ldots, s}$ are i.i.d. standard normal random variables. The boundary condition on the right-hand side is chosen such that the system of equations \eqref{eq:diff1d}--\eqref{eq:diff1d_rf} yields the explicit solution~\citep{MS13}
\[
u(x, \omega) = \int_0^x \frac{1}{a(z,\omega)} \Big(R + \int_z^1 f (w) \rd w\Big) \rd z.
\]
Specifically, for $R=1/2$ and the right-hand-side function $f(x) = x$, we obtain the solution 
\begin{equation}\label{eq:diff1d_soln}
u(x, \omega) = \int_0^x \frac{1 - z^2/2}{a(z,\omega)} \rd z.
\end{equation}
We define the QoI as the value of the potential function at $x = 0.5$, i.e.,  \[
F_1(\omega):=u(0.5, \omega).\] The fast implementation of \cite{Waldvogel} of the Clenshaw--Curtis quadrature rule \citep{CC60} is used to evaluate the integral in \eqref{eq:diff1d_soln} for each realisation of the random draw $\omega$, with absolute error tolerance $10^{-10}$. 

We numerically test the convergence of interlaced polynomial lattice point sets on a range of stochastic dimensions, $s \in \{2^2, 2^3, \ldots, 2^6 \}$. For each dimension, we use QMC point sets with increasing sizes $N \in \{2^{17}, 2^{18}, \ldots, 2^{22}\}$ to verify the convergence in integrating the QoI $F_1(\omega)$, comparing to the ``true'' integral calculated using $N = 2^{23}$ QMC points. Figure~\ref{fig:PDE1d_N} summaries the convergence of interlaced polynomial lattice point sets for each stochastic dimension. The observed convergence rates lie between $N^{-1}$ and $N^{-1.5}$ in this set of experiments. Overall, the convergence rate improves with increasing stochastic dimension. For $s=2^2$, the observed rate is about $N^{-1.09}$. For $s=2^{6}$, the observed rate is about $N^{-1.39}$. For each QMC sample size, we also report the estimated integration error as a function of the stochastic dimension. The results are shown in Figure~\ref{fig:PDE1d_s}. It appears that the integration error increases with stochastic dimension for all QMC sample sizes, which is in agreement with the QMC theory developed in this paper.

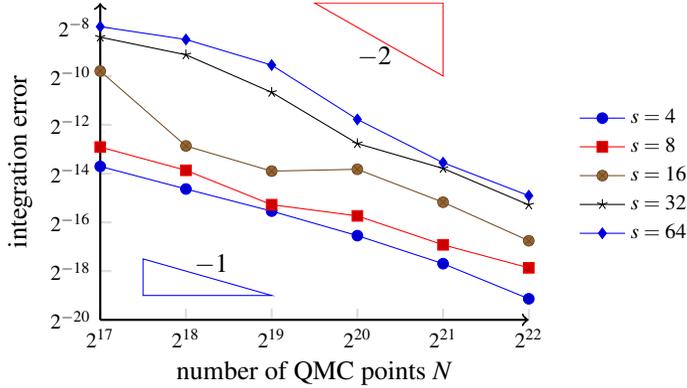
\begin{figure}
\centering
\begin{tikzpicture}
\begin{axis}[%
width=0.5\linewidth,
height=0.4\linewidth,
xmode=normal,
xmin=17,xmax=22,
ymode=normal,
ymin=-20,ymax=-7,
xlabel={number of QMC points $N$},
ylabel={integration error},
legend style={at={(1.4,0.7)},anchor=north east},
legend cell align={left},
nodes near coords,point meta = explicit symbolic,
x label style={at={(0.5,-0.1)}},
xtick={17,18,19,20,21,22},
xticklabels={$2^{17}$,$2^{18}$,$2^{19}$,$2^{20}$,$2^{21}$,$2^{22}$},
ytick={-6,-8,-10,-12,-14,-16,-18,-20},
yticklabels={$2^{-6}$,$2^{-8}$,$2^{-10}$,$2^{-12}$,$2^{-14}$,$2^{-16}$,$2^{-18}$,$2^{-20}$}
]
\addplot coordinates{
(16, -12.8507)
(17, -13.7062)
(18, -14.6325)
(19, -15.5405)
(20, -16.5451)
(21, -17.6958)
(22, -19.1362)
}; \addlegendentry{$s=4$};

\addplot coordinates{
(16, -11.6318)
(17, -12.9118)
(18, -13.8661)
(19, -15.2770)
(20, -15.7361)
(21, -16.9223)
(22, -17.8659)
}; \addlegendentry{$s=8$};

\addplot coordinates{
(16, -8.5644)
(17, -9.7969)
(18, -12.8666)
(19, -13.8995)
(20, -13.8240)
(21, -15.1757)
(22, -16.7558)
}; \addlegendentry{$s=16$};

\addplot coordinates{
(16, -7.4799)
(17, -8.4049)
(18, -9.1273)
(19, -10.6679)
(20, -12.7701)
(21, -13.7890)
(22, -15.2843)
}; \addlegendentry{$s=32$};

\addplot coordinates{
(16, -7.0306)
(17, -7.9741)
(18, -8.4999)
(19, -9.5502)
(20, -11.7776)
(21, -13.5521)
(22, -14.9070)
}; \addlegendentry{$s=64$};

\draw[blue] (axis cs:17.5,-17.5) -- (axis cs:17.5,-19) -- (axis cs:19,-19) -- cycle;
\node at (axis cs:18.3,-17.7) {$-1$};

\draw[red] (axis cs:19.5,-7) -- (axis cs:21,-7) -- (axis cs:21,-10) -- cycle;
\node at (axis cs:20.2,-9.2) {$-2$};

\end{axis}
\end{tikzpicture}
\caption{\label{fig:PDE1d_N} QMC integration errors for integrating the quantity of interest $F_1(\omega)$ in the one-dimensional PDE example, shown as a function of QMC sample size for stochastic dimensions $s \in \{2^2, 2^3, \ldots, 2^6\}$. The integration errors and QMC sample sizes are reported on the base $2$ logarithmic scale. The blue and red triangles indicate reference slopes of $-1$ and $-2$, respectively.}
\end{figure}

\begin{figure}
\centering
\begin{tikzpicture}
\begin{axis}[%
width=0.5\linewidth,
height=0.4\linewidth,
xmode=log,
log basis x=2, 
xmin=4,xmax=64,
ymode=normal,
ymin=-20,ymax=-7,
xlabel={dimension $s$},
ylabel={integration error},
legend style={at={(1.4,0.7)},anchor=north east},
legend cell align={left},
nodes near coords,point meta = explicit symbolic,
x label style={at={(0.5,-0.1)}},
xtick={4,8,16,32,64},
ytick={-6,-8,-10,-12,-14,-16,-18,-20},
yticklabels={$2^{-6}$,$2^{-8}$,$2^{-10}$,$2^{-12}$,$2^{-14}$,$2^{-16}$,$2^{-18}$,$2^{-20}$}
]

\addplot coordinates{
(64, -7.9741)
(32, -8.4049)
(16, -9.7969)
(8, -12.9118)
(4, -13.7062)
}; \addlegendentry{$N=2^{17}$};

\addplot coordinates{
(64, -8.4999)
(32, -9.1273)
(16, -12.8666)
(8, -13.8661)
(4, -14.6325)
}; \addlegendentry{$N=2^{18}$};

\addplot coordinates{
(64, -9.5502)
(32, -10.6679)
(16, -13.8995)
(8, -15.2770)
(4, -15.5405)
}; \addlegendentry{$N=2^{19}$};

\addplot coordinates{
(64, -11.7776)
(32, -12.7701)
(16, -13.8240)
(8, -15.7361)
(4, -16.5451)
}; \addlegendentry{$N=2^{19}$};

\addplot coordinates{
(64, -11.7776)
(32, -12.7701)
(16, -13.8240)
(8, -15.7361)
(4, -16.5451)
}; \addlegendentry{$N=2^{20}$};

\addplot coordinates{
(64, -13.5521)
(32, -13.7890)
(16, -15.1757)
(8, -16.9223)
(4, -17.6958)
}; \addlegendentry{$N=2^{21}$};

\addplot coordinates{
(64, -14.9070)
(32, -15.2843)
(16, -16.7558)
(8, -17.8659)
(4, -19.1362)
}; \addlegendentry{$N=2^{22}$};

\draw[blue] (axis cs:22.63,-19) -- (axis cs:45.25,-19) -- (axis cs:45.25,-18) -- cycle;
\node at (axis cs:32,-18) {$1$};

\draw[red] (axis cs:5.66,-8) -- (axis cs:11.31,-8) -- (axis cs:5.66,-10) -- cycle;
\node at (axis cs:8.2,-9.5) {$2$};

\end{axis}
\end{tikzpicture}
\caption{\label{fig:PDE1d_s} QMC integration errors for integrating the quantity of interest $F_1(\omega)$ in the one-dimensional PDE example, shown as a function of stochastic dimension for QMC sample sizes $N \in \{2^{17}, 2^{18}, \ldots, 2^{22}\}$. The integration errors and stochastic dimensions are reported on the base $2$ logarithmic scale. The blue and red triangles indicate reference slopes of $1$ and $2$, respectively.}
\end{figure}
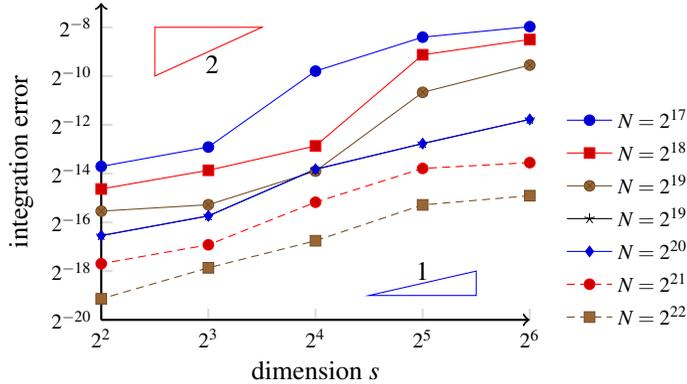

The second experiment considers the diffusion problem on a two-dimensional domain $D = [0,1]^2$:
\begin{equation}\label{def:diff2d}
-\nabla_{\bsx} \cdot (a(\bsx, \omega) \nabla_{\bsx} \, u(\bsx, \omega)) = f(\bsx), \quad \mbox{for almost all $ \omega \in \Omega$ and $\bsx \in D$,}
\end{equation}
subject to Dirichlet boundary conditions
\begin{equation}\label{def:diff2d_bc1}
u(\bsx,\omega) = 1 \text{ for } \bsx \in \{(0,x_2): x_2 \in [0,1]\}, \quad u(\bsx,\omega) = 0 \text{ for } \bsx \in \{(1,x_2): x_2 \in [0,1]\}
\end{equation}
on the left and right boundaries, and the no-flux boundary condition
\begin{equation}\label{def:diff2d_bc2}
a(\bsx, \omega) \nabla_{\bsx} u(\bsx,\omega) \cdot \vec{n}(\bsx) = 0 \text{ for } \bsx \in \{(x_1,0): x_1 \in [0,1]\} \text{ and } \bsx \in \{(x_1,1): x_1 \in [0,1]\}
\end{equation}
on the top and bottom boundaries, where $\vec{n}(\bsx)$ is the outward normal vector along the boundary. We specify the forcing term $f(\bsx)$ as a bimodal function defined as a superposition of Gaussian kernels
\begin{equation}\label{def:diff2d_forcing}
f(\bsx) = \frac{1}{2\pi \, (0.05)^2} \exp\!\left( -\frac{1}{2(0.05)^2} \big\| \bsx - \bsx_1\big\|^2 \right) + \frac{1}{2\pi \, (0.05)^2} \exp\!\left( -\frac{1}{2(0.05)^2} \big\| \bsx - \bsx_2\big\|^2 \right),
\end{equation}
with centers $\bsx_1 = (0.5,\,0.9)$ and $ \bsx_2 = (0.5,\,0.1)$. The log-normal random field $a(\bsx, \omega)$ is defined as
\begin{equation}
a(\bsx, \omega) = \exp\Big( \sum_{j = 1}^s \sqrt{\mu_j} \, \xi_j(\bsx) Y_j(\omega) \Big)
\end{equation}
where $\big(Y_j(\omega)\big)_{j = 1, \ldots, s}$ are i.i.d. standard normal random variables, and $(\mu_j, \xi_j(\bsx))_{j=1, \ldots, s}$ are the leading eigenvalues and eigenfunctions of the covariance operator defined by the Mat\'ern kernel 
\[
k(\bsx,\bsy) =\frac{2^{1-\nu}}{\Gamma(\nu)} 
\left( \sqrt{2\nu} \|\bsx - \bsy\| \right)^{\nu}
K_{\nu}\!\left( \sqrt{2\nu}\,\|\bsx - \bsy\| \right)
\]
with $\nu = 1$, where $\Gamma$ is the Gamma function and $K_{\nu}$ is the modified Bessel function of the second kind. Here, we truncate the eigen expansion at dimension $s=32$. The quantity of interest is the integral of the potential function over the domain $D$:
\begin{equation}\label{def:diff2d_qoi}
F_2(\omega) = \int_D u(\bsx,\omega) \mathrm{d} \bsx.
\end{equation}
We solve the system of equations \eqref{def:diff2d}--\eqref{def:diff2d_qoi} using the finite element method with piecewise bilinear basis functions on a uniform grid of mesh size $h = 1/32$.

We use QMC point sets with increasing sizes $N \in \{2^{17}, 2^{18}, \ldots, 2^{21}\}$ to verify the convergence in integrating the QoI $F_2(\omega)$ in this example, compared to the ``true'' integral calculated using $N = 2^{22}$ QMC points. The estimated integration error is reported in Figure~\ref{fig:2dpde}. In this case, the observed convergence rate is about $N^{-2.02}$, which is slightly better than our theoretical prediction. 

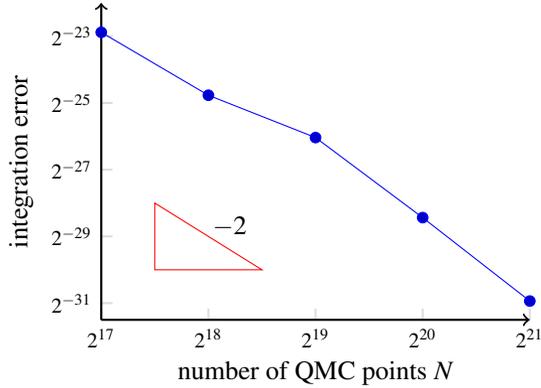
\begin{figure}
\centering
\begin{tikzpicture}
\begin{axis}[%
width=0.5\linewidth,
height=0.4\linewidth,
xmode=normal,
xmin=17,xmax=21,
ymode=normal,
ymin=-31.5,ymax=-22,
xlabel={number of QMC points $N$},
ylabel={integration error},
legend style={at={(0.98,0.02)},anchor=south east},
nodes near coords,point meta = explicit symbolic,
x label style={at={(0.5,-0.1)}},
xtick={17,18,19,20,21},
xticklabels={$2^{17}$,$2^{18}$,$2^{19}$,$2^{20}$,$2^{21}$},
ytick={-23,-25,-27,-29,-31},
yticklabels={$2^{-23}$,$2^{-25}$,$2^{-27}$,$2^{-29}$,$2^{-31}$}
]
\addplot coordinates{
   (17,  -22.8822)
   (18,  -24.7679)
   (19.,  -26.0364)
   (20,  -28.4360)
   (21,  -30.9374)
}; 
\draw[red] (axis cs:17.5,-28) -- (axis cs:17.5,-30) -- (axis cs:18.5,-30) -- cycle;
\node at (axis cs:18.2,-28.7) {$-2$};
\end{axis}
\end{tikzpicture}
\caption{QMC integration errors for integrating the quantity of interest $F_2(\omega)$ in the two-dimensional PDE example, shown as a function of QMC sample size. The integration errors and QMC sample sizes are reported on the base $2$ logarithmic scale. The triangle indicates a reference slope of $-2$. The stochastic dimension is $s=32$ in this example. }\label{fig:2dpde}
\end{figure}

\begin{appendix}
\section{Proof of Lemma~\ref{leWal}}\label{app:A}    

\begin{proof}
Our $k$ is of the form $k=\kappa_0+\kappa_1+\cdots+\kappa_{r-1} 2^{r-1}+2^r$ with digits $\kappa_0,\kappa_1,\ldots,\kappa_{r-1}\in \{0,1\}$. Then $$\wal_k(t)=(-1)^{\kappa_0 t_1+\kappa_1 t_2+\cdots +\kappa_{r-1} t_r+t_{r+1}} \quad \mbox{for $t=t_1 2^{-1}+t_2 2^{-2}+\cdots \in [0,1)$.}$$ In particular, $\wal_k(t)=\wal_k(\frac{m}{2^r})$ if $t \in [\frac{m}{2^r},\frac{m+1/2}{2^r})$ and $\wal_k(t)=-\wal_k(\frac{m}{2^r})$ if $t \in [\frac{m+1/2}{2^r},\frac{m+1}{2^r})$ for $m \in \{0,1,\ldots,2^r-1\}$. This implies that for every $m \in \{0,1,\ldots,2^r-1\}$ we have $\int_{I(m,r)} \wal_k(t) \rd t=0$, where $I(m,r):=[\frac{m}{2^r},\frac{m+1}{2^r})$. Now, let $a \in \{0,1,\ldots,2^r-1\}$ be such that $x \in I(a,r)$. We consider two cases:
\begin{enumerate}
\item If $\frac{a}{2^r} \le x < \frac{a+1/2}{2^r}$, then $2^r x-a \in [0,\tfrac{1}{2})$ and hence $\eta(2^rx)=\eta(2^r x-a )=2^r x-a$. Furthermore, $\wal_k(x)=\wal_k(\frac{a}{2^r})$. Then we have $$\int_0^x \wal_k(t)\rd t = \int_{a/2^r}^x \wal_k\left(\frac{a}{2^r}\right) \rd t = \wal_k\left(\frac{a}{2^r}\right) \left(x-\frac{a}{2^r}\right)=\frac{1}{2^r} \, \wal_k(x) \, \eta(2^r x).$$
\item If $\frac{a+1/2}{2^r} \le x < \frac{a+1}{2^r}$, then $2^r x - a \in [\frac{1}{2},1)$ and hence $\eta(2^rx)=\eta(2^r x-a)=2^r x-a-1$. Furthermore, $\wal_k(x)=-\wal_k(\frac{a}{2^r})$. Then we have  
\begin{align*}
\int_0^x \wal_k(t)\rd t = & \int_{a/2^r}^{(a+(1/2))/2^r} \wal_k\left(\frac{a}{2^r}\right) \rd t - \int_{(a+(1/2))/2^r}^x  \wal_k\left(\frac{a}{2^r}\right) \rd t  \\
= & \wal_k\left(\frac{a}{2^r}\right) \left( \frac{1}{2^{r+1}}-x+\frac{a}{2^r}+\frac{1}{2^{r+1}}\right)\\
= & \frac{1}{2^r} \, \wal_k(x) \, \eta(2^r x).
\end{align*}
\end{enumerate}
This shows the first part of the lemma. The assertion about the supremum follows immediately.
\end{proof}

\section{Proof of Lemma~\ref{le:supp:W}}\label{app:B}   

\begin{proof}
Let $k = 2^{a_1} + \ell$ with $\ell \in \{0,\ldots, 2^{a_1}-1\}$. For $x \in [0,1)\setminus I(m,a_1)$ we have
\begin{equation*}
\int_0^x \bsone_{I(m,a_1)}(y) \wal_k(y) \rd y = 0,
\end{equation*}
and for $x \in I(m,a_1)$ we have 
\begin{align*}
\int_0^x \bsone_{I(m,a_1)}(y) \wal_k(y) \rd y = & \int_{m/2^{a_1}}^x \wal_k(y) \rd y \\
= & \int_0^x \wal_k(y) \rd y - \int_0^{m/2^{a_1}} \wal_k(y) \rd y =  \frac{1}{2^{a_1}} \wal_k(x) \eta(2^{a_1}x),
\end{align*}
where we used Lemma~\ref{leWal}. Put
\begin{equation*}
P_a(x) := \begin{cases} x - \frac{r}{2^a} & \mbox{for } x \in \left[\frac{r}{2^a}, \frac{r+1/2}{2^a}\right) \mbox{ for some } r \in \{0, 1, \ldots, 2^a\}, \\[0,5em] \frac{r+1}{2^a} - x & \mbox{for } x \in \left[\frac{r+1/2}{2^a}, \frac{r+1}{2^a} \right) \mbox{ for some } r \in \{0, 1, \ldots, 2^a\}. \end{cases}
\end{equation*}
If $x \in [\frac{m}{2^{a_1}},\frac{m+1/2}{2^{a_1}})$, then $$\frac{1}{2^{a_1}} \wal_k(x) \eta(2^{a_1}x) =\wal_{\ell}(x) \left(x-\frac{m}{2^{a_1}}\right)=\wal_{\ell}(x) P_{a_1}(x).$$  If $x \in [\frac{m+1/2}{2^{a_1}},\frac{m+1}{2^{a_1}})$, then $$\frac{1}{2^{a_1}} \wal_k(x) \eta(2^{a_1}x) =-\wal_{\ell}(x) \left(x-\frac{m+1}{2^{a_1}}\right)=\wal_{\ell}(x) P_{a_1}(x).$$
Together, we have
\begin{equation*}
\int_0^x \bsone_{I(m,a_1)}(y) \wal_k(y) \rd y = \wal_{\ell}(x) P_{a_1}(x) \bsone_{I(m,a_1)}(x).
\end{equation*}

For $a \in \{0,\ldots,a_1\}$ we have $\bsone_{I(m,a)}=\sum_{m'=m 2^{a_1-a}}^{(m+1) 2^{a_1-a}-1} \bsone_{I(m',a_1)}$ and hence
\begin{equation*}
\int_0^x \bsone_{I(m,a)}(y) \wal_k(y) \rd y = \wal_{\ell}(x) P_{a_1}(x) \bsone_{I(m,a)}(x).
\end{equation*}

Now, let $k \in N_{2}\setminus N_{1}$ be of the form $k=2^{a_1}+2^{a_{2}}$. Then 
\begin{align*}
W_{k,m,2}(x) = & 2^{a_2 /2} \int_0^{x} \boldsymbol{1}_{I(m,a_2)}(y) \wal_k(y) \rd y\\
= & 2^{a_2 /2} \wal_{2^{a_2}}(x) P_{a_1}(x) \bsone_{I(m,a_{2})}(x) = w_{2^{a_2},m,1}(x) P_{a_1}(x).
\end{align*}
Since the function $w_{2^{a_2},m,1}$ is zero in $[0, 1) \setminus I(m,a_2)$, the same holds true for $W_{k,m,2}$.

It is obvious that $\widetilde{W}_{k,m,2}(x)=0$ if $x < \frac{m}{2^{a_2}}$. If $x \ge \frac{m+1}{2^{a_2}}$, then
\begin{align*}
\widetilde{W}_{k,m,2}(x) = & \int_{m/2^{a_2}}^{(m+1)/2^{a_2}} 2^{a_2/2} \wal_{2^{a_2}}(y) P_{a_1}(y)\rd y\\
= & \int_{m/2^{a_2}}^{(m+1/2)/2^{a_2}} 2^{a_2/2} P_{a_1}(y)\rd y - \int_{(m+1/2)/2^{a_2}}^{(m+1)/2^{a_2}} 2^{a_2/2} P_{a_1}(y)\rd y \\
= & 2^{a_2/2} \sum_{r=m 2^{a_1-a_2}}^{(m+1/2) 2^{a_1-a_2}-1} \int_{r/2^{a_1}}^{(r+1)/2^{a_1}} P_{a_1}(y)\rd y - 2^{a_2/2} \sum_{r=(m+1/2) 2^{a_1-a_2}}^{(m+1) 2^{a_1-a_2}-1} \int_{r/2^{a_1}}^{(r+1)/2^{a_1}} P_{a_1}(y)\rd y \\
= & 2^{a_2/2} \left(\sum_{r=m 2^{a_1-a_2}}^{(m+1/2) 2^{a_1-a_2}-1} \frac{1}{2^{2a_1+2}} -  \sum_{r=(m+1/2) 2^{a_1-a_2}}^{(m+1) 2^{a_1-a_2}-1} \frac{1}{2^{2a_1+2}}\right)\\
= & 0,
\end{align*}
as claimed.
\end{proof}

\end{appendix}

\bibliographystyle{abbrvnat} 
\bibliography{ref}

\end{document}